\documentclass[english,11pt,a4paper,leqno]{amsart}

\usepackage[T1]{fontenc}
\usepackage[francais]{babel}
\usepackage[utf8]{inputenc}
\usepackage[mathscr]{euscript}
\usepackage{textcomp,multicol,enumitem}
\usepackage[Bjornstrup]{fncychap}

\usepackage{amssymb,url,amsmath}

\theoremstyle{plain}\newtheorem{theorem}{Th\'eor\`eme}

\theoremstyle{plain}\newtheorem{theoreme}{Th\'eor\`eme}[section]
\theoremstyle{plain}\newtheorem{theo}{Th\'eor\`eme}
\theoremstyle{plain}\newtheorem{prop}[theoreme]{Proposition}
\theoremstyle{plain}\newtheorem{lem}[theoreme]{Lemme}
\theoremstyle{plain}
\theoremstyle{plain}
\theoremstyle{plain}
\theoremstyle{definition}\newtheorem{defi}[theoreme]{D\'efinition}
\theoremstyle{definition}\newtheorem{remarque}[theoreme]{Remarque}

\newcommand{\BibTeX}{{\scshape Bib}\kern-.08em\TeX}
\newcommand{\T}{\S\kern .15em\relax }
\newcommand{\AMS}{$\mathcal{A}$\kern-.1667em\lower.5ex\hbox
	{$\mathcal{M}$}\kern-.125em$\mathcal{S}$}

\title[MUM, congruences "à la Lucas" et indépendance algébrique]{Monodromie unipotente maximale, congruences "à la Lucas" et indépendance algébrique}
\date {}
\author{Daniel Vargas Montoya}

\address{Institut Camille Jordan, Universit\'e Claude Bernard Lyon 1, Batîment Braconnier, 21 Avenue Claude Bernard, 69100 Villeurbanne\\
}
\email{vargas@math.univ-lyon1.fr}
\thanks{This project has received funding from the European Research Council (ERC) under the European Union's Horizon 2020 research and innovation programme under the Grant Agreement No 648132. }

\begin{document}

\begin{abstract}
 Soient $f(z)\in 1+z\mathbb{Q}[[z]]$ et $\mathcal{S}$ un ensemble infini de nombres premiers tels que, pour tout $p\in\mathcal{S}$, nous pouvons réduire $f(z)$ modulo $p$. 
 Lorsque $f(z)$ est holonome, on obtient g\'en\'eralement que $f(z)_{\mid p}$  est alg\'ebrique sur $\mathbb F_p(z)$. Si de plus  $f(z)_{\mid p}$  annule un polyn\^ome 
 de la forme $X-A_p(z)X^{p^l}$, on peut utiliser ces \'equations pour obtenir des r\'esultats de transcendance et d'ind\'ependance alg\'ebrique sur $\mathbb Q(z)$. 
 Dans cet article, nous cherchons des conditions sur les op\'erateurs diff\'erentiels annulant $f(z)$ qui  garantissent l'existence de ces \'equations particuli\`eres. 
 Supposons que $f(z)$ annule un opérateur différentiel $\mathcal{H}\in\mathbb{Q}(z)[d/dz]$ muni d'une structure de Frobenius forte pour tout $p\in\mathcal{S}$  ainsi qu'un opérateur différentiel fuchsien $\mathcal{D}\in\mathbb{Q}(z)[d/dz]$ tel que zéro est un point singulier régulier de $\mathcal{D}$ et les exposants en zéro de $\mathcal{D}$ sont tous égaux à zéro. Notre résultat principal établit que pour presque tout $p\in\mathcal{S}$, $f(z)_{\mid p}$  annule un polynôme de la forme $X-A_p(z)X^{p^l}$, où $A_p(z)$ est une fraction rationnelle à coefficients dans $\mathbb{F}_p$ de hauteur inférieure ou égale à $Cp^{2l}$ et $C$ est une constante strictement positive indépendante de $p$. Nous étudions aussi l'indépendance algébrique sur $\mathbb{Q}(z)$ de ces séries. 
\end{abstract}

\maketitle
\tableofcontents

\section{Introduction}

Soit $p$ un nombre premier. 
Dans cet article, nous suivrons \cite{allouche} en disant qu'une série formelle $f(z)=\sum_{n\geq0}a(n)z^n$ à coefficients dans $\mathbb{Q}$ est \emph{$p$-Lucas} si $f(z)\in\mathbb{Z}_{(p)}[[z]]$, où $\mathbb{Z}_{(p)}$ est la localisation de $\mathbb{Z}$ en l'idéal $(p)$, $a(0)=1$ et pour tout entier positif $m$ et tout $r\in\{0,\ldots, p-1\}$ on a 
$$a(r+mp)\equiv a(r)a(m)\mod p\mathbb{Z}_{(p)} \,.$$ 
Une observation importante est que $f(z)=\sum_{n\geq0}a(n)z^n\in1+z\mathbb{Z}_{(p)}[[z]]$ est $p$-Lucas si et seulement si 
\begin{equation}\label{eq: plucas}
f_{\mid p}(z)=A_p(z)f_{\mid p}(z^p)\,,
\end{equation} 
où $f_{\mid p}(z)=\sum_{n\geq0}(a(n)\mod p)z^n$ d\'esigne la r\'eduction modulo $p$ de $f(z)$ et $A_p(z)=\sum_{n=0}^{p-1}(a(n)\mod p)z^n$. 
Puisque $f_{\mid p}(z^p)=f_{\mid p}(z)^p$, la s\'erie formelle $f_{\mid p}(z)$ est alors alg\'ebrique sur $\mathbb F_p(z)$ de degr\'e au plus $p-1$. 
On trouve dans la litt\'erature de nombreux exemples de s\'eries 
formelles v\'erifiant de telles congruences pour tout nombre premier $p$, ou au moins pour une infinit\'e d'entre eux.  
Par exemple, le th\'eor\`eme de Lucas sur les coefficients binomiaux implique que, pour tout entier $r\geq 1$, la s\'erie formelle 
$$
\mathfrak{g}_r(z)= \sum_{n\geq0}\binom{2n}{n}^r z^n
$$
est $p$-Lucas pour tout nombre premier $p$. 
Gessel \cite{gessel} a montr\'e que c'est \'egalement le cas de la série génératrice des nombres d'Apéry $$\mathfrak{t}(z)=\sum_{n\geq0}\left(\sum_{k=0}^n\binom{n}{k}^2\binom{n+k}{k}^2\right)z^n\,.$$
Nous renvoyons \`a \cite{apery,lucas,Borisgfonct} pour davantage d'exemples.  
Notons que l'approche utilisée dans \cite{gessel}, \cite{apery} et \cite{lucas} est combinatoire, 
tandis que les auteurs de \cite{Borisgfonct} \'etudie la valuation $p$-adique des coefficients des séries hypergéométriques généralisées et de quotients de factorielles de plusieurs variables.  

En 1989, Sharif et Woodcock \cite{transcedencia} ont remarqu\'e que les \'equations~\eqref{eq: plucas} peuvent \^etre utilis\'ees pour obtenir la transcendance 
sur $\mathbb Q(z)$ de certaines s\'eries $p$-Lucas, en montrant que leurs r\'eductions modulo $p$ sont des s\'eries formelles alg\'ebriques sur $\mathbb F_p(z)$ dont le degr\'e n'est pas born\'e en fonction de $p$. Ils d\'emontrent ainsi la transcendance des s\'eries $\mathfrak g_r(z)$ pour $r\geq 2$\footnote{Notons que $\mathfrak g_1(z)$ est alg\'ebrique.}.  
Cette approche a ensuite \'et\'e \'etendue par Allouche, Gouyou-Beauchamps et Skordev dans \cite{allouche}.  Plus r\'ecemment, Adamczewski et Bell \cite{AB13} ont 
montr\'e comment utiliser ces m\^emes \'equations pour d\'emontrer l'ind\'ependance alg\'ebrique sur $\mathbb Q(z)$ de certaines s\'eries $p$-Lucas. 
Cette approche a ensuite \'et\'e d\'evelopp\'ee par Adamczewski, Bell et Delaygue dans \cite{Borisgfonct}. Ces auteurs ont  introduit de nouveaux ensembles de s\'eries formelles, les ensembles $\mathcal{L}(\mathcal{S})$, qui g\'en\'eralisent les ensembles de s\'eries $p$-Lucas, et ils ont donn\'e un critère d'indépendance algébrique pour les \'el\'ements de $\mathcal{L}(\mathcal{S})$. Ce crit\`ere permet par exemple de d\'emontrer que les s\'eries $\mathfrak{g}_r$, $r\geq 2$, et $\mathfrak{t}(z)$ 
sont alg\'ebriquement ind\'ependantes sur $\mathbb Q(z)$. 

\begin{defi}\label{pluca}
	Pour un ensemble $\mathcal{S}$ infini de nombres premiers, $\mathcal{L}(\mathcal{S})$ est l'ensemble des séries $f(z)\in 1+z\mathbb{Q}[[z]]$ telles que pour tout $p\in\mathcal{S}$:
	\begin{enumerate}[label=\alph*)]
		\item $f(z)\in\mathbb{Z}_{(p)}[[z]]$.
		\item Il existe un entier $l_p$ strictement positif et une fraction rationnelle $A_p(z)$ dans $\mathbb{F}_p(z)\cap\mathbb{F}_p[[z]]$ tels que
		\begin{equation}\label{àdeux}
		f_{\mid p}(z)=A_p(z)f_{\mid p}(z^{p^{l_p}}),
		\end{equation}
		où $f_{\mid p}(z)$ est la réduction de $f(z)$ modulo $p$.  
		\item  La hauteur\footnote{Si nous écrivons $A_p(z)=P(z)/Q(z)$, où $P(z),Q(z)$ sont des polynômes premiers entre eux, la hauteur de $A_p(z)$ est le maximum des degrés de $P(z)$ et $Q(z)$.} de $A_p(z)$ est inférieure ou égale à $Cp^{l_p}$, où $C$ est une constante qui ne dépend pas de $p$. 
	\end{enumerate}
\end{defi}
Il est clair que si, pour tout $p\in\mathcal{S}$, la série $f(z)$ est $p$-Lucas, alors $f(z)\in\mathcal{L}(\mathcal{S})$. Le crit\`ere d'ind\'ependance alg\'ebrique 
d'Adamczewski, Bell et Delaygue s'\'ennonce de la fa\c con suivante. 

\begin{theorem}[Théorème 1.3 de \cite{Borisgfonct}]\label{5.1}
	Soit $\mathcal{S}$ un ensemble infini de nombres premiers. Des s\'eries formelles $f_1(z),\ldots, f_r(z)\in\mathcal{L}(\mathcal{S})$ 
	sont algébriquement dépendantes sur $\mathbb{Q}(z)$ si et seulement si il existe $m_1,\ldots, m_r\in\mathbb{Z}$, non tous nuls, tels que $f_1(z)^{m_1}\cdots f_r(z)^{m_r}\in\mathbb{Q}(z).$ 	
\end{theorem} 

Bien que cela n'apparaisse pas dans les d\'efinitions correspondantes, le cadre naturel d'\'etude des s\'eries $p$-Lucas et des \'el\'ements de $\mathcal L(\mathcal S)$ est celui des 
s\'eries holonomes, c'est-\`a-dire des solutions d'\'equations diff\'erentielles lin\'eaires \`a coefficients polyn\^omes. En effet, les exemples connus d'\'el\'ements de  $\mathcal L(\mathcal S)$ sont dans une immense majorit\'e de ce type (voir toutefois l'exemple donn\'e \`a la fin de \cite{allouche} pour une exception \`a cette r\`egle). 
Soit $f(z)\in \mathbb{Q}[[z]]$ holonome telle que, pour tout $p\in\mathcal{S}$, la série $f(z)$ est dans $\mathbb{Z}_{(p)}[[z]]$. Motiv\'e par le th\'eor\`eme A, 
nous cherchons dans cet article des conditions sur les op\'erateurs diff\'erentiels annulant $f(z)$ qui garantissent  que $f(z)$ appartient à $\mathcal{L}(\mathcal{S})$. 

Notons d'abord que, d'apr\`es l'équation \eqref{àdeux}, si  $f(z) \in\mathcal{L}(\mathcal{S})$, alors $f_{\mid p}(z)$ est algébrique 
sur $\mathbb{F}_p(z)$ pour tout $p\in\mathcal{S}$. Cela nous conduit \`a faire appel à la notion, introduite par Dwork dans \cite{dworksFf}, 
de \emph{structure de Frobenius forte} associ\'ee \`a  un opérateur différentiel et un nombre premier $p$ (voir d\'efinition~\ref{def:SFF}). Rappelons que les opérateurs de \emph{Picard-Fuchs} sont des exemples d'opérateurs qui possèdent une structure de Frobenius forte pour presque tout nombre premier $p$, ce qui montre la pertinence de cette notion pour notre \'etude. 

\begin{defi}\label{FS}
Pour un ensemble $\mathcal S$ de nombres premiers,  $\mathcal{F}(\mathcal{S})$ est l'ensembles des séries $f(z)\in 1+z\mathbb{Q}[[z]]$ telles que: 
\begin{itemize}
	\item pour tout $p\in\mathcal{S}$, $f(z)\in\mathbb{Z}_{(p)}[[z]]$;
	\item $f(z)$ annule un opérateur différentiel $\mathcal{H}\in\mathbb{Q}(z)[\delta]$ muni d'une structure de Frobenius forte pour tout $p\in\mathcal{S}$, où $\delta=zd/dz$. 
\end{itemize}
\end{defi}

Les travaux de Christol \cite{Gillesalgebriques} 
(voir aussi \cite[Th\'eor\`eme 2.6]{vmsff}) implique que si $f(z)\in\mathcal{F}(\mathcal{S})$ alors la série $f_{\mid p}(z)$ est algébrique sur $\mathbb{F}_p(z)$ pour tout $p\in\mathcal{S}$. Nous cherchons donc une condition suppl\'ementaire \`a ajouter aux \'el\'ements de $\mathcal{F}(\mathcal{S})$ afin de garantir leur appartenance \`a 
$\mathcal L(\mathcal S)$. Dans cette direction, le  théorème~\ref{practico} établit que si $f(z)\in\mathcal{F}(\mathcal{S})$ annule un opérateur différentiel $\mathcal{D}\in\mathbb{Q}(z)[d/dz]$ MOM en zéro,\footnote{Le lecteur trouvera la définition d'opérateur différentiel MOM en zéro dans la partie~\ref{opérateurs différentiels}.} alors la série $f(z)$ vérifie une équation du type~\eqref{àdeux} pour presque tout $p\in\mathcal{S}$. 
Par contre, et de fa\c con assez surprenante, la condition $c$ de la d\'efinition des ensembles $\mathcal L(\mathcal S)$ n'est pas toujours v\'erifi\'ee. 
On obtient seulement que la hauteur de $A_p(z)$ est inférieure ou égale à $Cp^{2l_p}$, où $C$ est une constante indépendante de $p$.
La deuxième partie du théorème~\ref{practico} donne ensuite une condition suffisante pour que $f(z)$ appartienne à $\mathcal{L}(\mathcal{S})$. 

Avant d'énoncer le théorème~\ref{practico}, nous introduisons l'ensemble $\mathcal{L}^2(\mathcal{S})$ et les op\'erateurs $\Lambda_p$.

\begin{defi}\label{L2}
	Soit $\mathcal{S}$ un ensemble infini de nombres premiers.  L'ensemble $\mathcal{L}^2(\mathcal{S})$ est constitué des séries $f(z)\in 1+z\mathbb{Q}[[z]]$ telles que, pour tout $p\in\mathcal{S}$, les conditions a) et b) de la définition~\ref{pluca} sont vérifiées et
	
	c') la hauteur de $A_p(z)$ est inférieure ou égale à $Cp^{2l_p}$, où $C$ est une constante qui ne dépend pas de $p$. 
\end{defi}
Il est clair que $\mathcal{L}(\mathcal{S})\subset\mathcal{L}^2(\mathcal{S})$.

\begin{defi}
	Soient $p$ un nombre premier et $K$ un corps quelconque. Nous notons $\Lambda_p$ l'opérateur de $K[[z]]$  
	défini par: $\Lambda_p(\sum_{n\geq0}a(n)z^n)=\sum_{n\geq0}a(np)z^n$. Pour un entier $k\geq1$, $\Lambda^k_p=\Lambda_p\circ\cdots\circ\Lambda_p$ $k$-fois et $\Lambda^0_p=Id$ est l'opérateur identité.
\end{defi}

Par abus de langage, les op\'erateurs $\Lambda_p$ seront appel\'es op\'erateurs de Cartier dans la suite\footnote{Voir la discussion dans \cite[Section 2]{AY} concernant cet abus de langage.}.

\begin{theo}\label{practico} Soient $\mathcal{S}$ un ensemble infini de nombres premiers et $f(z)$ une série qui appartient à $\mathcal{F}(\mathcal{S})$. 
	\begin{enumerate}
		\item[${\rm (1)}$] Si $f(z)$ annule un opérateur différentiel $\mathcal{D}\in\mathbb{Q}(z)[d/dz]$ MOM en zéro alors $f(z)\in\mathcal{L}^2(\mathcal{S}')$, où $\mathcal{S}'\subset\mathcal{S}$ et $\mathcal{S}\setminus\mathcal{S}'$ est fini.
		\item[${\rm (2)}$]  Si de plus, pour tout $p\in\mathcal{S}$, il existe un entier $l_p>0$ tel que $\Lambda^{l_p}_p(f_{\mid p}(z))=f_{\mid p}(z)$ alors $f(z)\in\mathcal{L}(\mathcal{S}')$, où $\mathcal{S}'\subset\mathcal{S}$ et $\mathcal{S}\setminus\mathcal{S}'$ est fini.
	\end{enumerate}	
\end{theo}

Malgr\'e son apparente complexit\'e, ce th\'eor\`eme fournit un moyen tr\`es efficace pour montrer que des s\'eries formelles appartiennent \`a $\mathcal{L}(\mathcal{S})$. 
On peut par exemple l'utiliser pour montrer que les s\'eries formelles $\mathfrak{g}_r(z)$, $r\geq 1$, et $\mathfrak{t}(z)$ sont dans $\mathcal{L}(\mathcal{P}\setminus\mathcal{J})$, où 
$\mathcal{P}$ est l'ensemble des nombres premiers et $\mathcal J$ est un ensemble fini de nombres premiers. De fa\c con plus remarquable, 
ce r\'esultat permet d'\'etudier les solutions d'op\'erateurs de type Calabi-Yau. Dans \cite{calabi}, les auteurs dressent une liste de plus de 400 opérateurs de type Calabi-Yau. 
Ces opérateurs vérifient certaines conditions algébriques dont: zéro est un point singulier régulier et les exposants en z\'ero sont tous égaux à zéro et chaque opérateur admet une solution dans $\mathbb{Z}[[z]]$ dont le terme constant est \'egal \`a 1. Notamment, tous ces opérateur sont MOM en zéro. 
Dans la partie~\ref{eqcalabi}, nous montrons que 242 de ces séries appartiennent à $\mathcal{L}(\mathcal{P}\setminus\mathcal J)$, où $\mathcal{J}$ est un ensemble fini de nombres premiers.

Dans la partie~\ref{exemple} nous montrons via un exemple que le fait que $f(z)$ annule un opérateur différentiel MOM en zéro n'est pas une condition suffisante pour que $f(z)$ soit dans $\mathcal{L}(S)$. Plus précisément, nous montrerons que la série hypergéométrique $$\mathfrak{f}_2(z)=\sum_{n\geq0}\frac{-1}{2n-1}\binom{2n}{n}^2z^n$$ n'appartient pas à $\mathcal{L}(\mathcal{S})$ quel que soit l'ensemble infini $\mathcal{S}$ de nombres premiers alors que $\mathfrak{f}_2(z)$ annule un opérateur différentiel MOM en zéro. On a toutefois que $\mathfrak{f}_2(z)\in\mathcal{L}^2(\mathcal{P}\setminus\{2\})$.  
Un autre problème intéressant à étudier est celui de l'indépendance algébrique sur $\mathbb{Q}(z)$ des séries qui appartiennent à $\mathcal{L}^2(\mathcal{S})$. Bien que nous ne donnions pas une réponse générale à ce problème, les résultats que nous obtenons dans la partie~\ref{independaalgebrique} nous permettent d'étudier l'indépendance algébrique de certaines séries qui sont dans $\mathcal{L}^2(\mathcal{S})$.  

\begin{theo}\label{independece}
	Les séries hypergéométriques $\{\mathfrak{f}_r\}_{r\geq2}$, où $$\mathfrak{f}_r(z)=\sum_{n\geq0}\frac{-1}{(2n-1)}\binom{2n}{n}^rz^n,$$ sont algébriquement indépendantes sur $\mathbb{Q}(z)$.
\end{theo}

La pertinence de ce r\'esultat vient du fait que le théorème~\ref{5.1} ne peut pas être appliqué aux séries $\mathfrak{f}_r(z)$.

Cet article est organisé de la manière suivante. Dans la partie~\ref{opérateurs différentiels} nous rappelons la définition d'opérateur différentiel fuchsien et définissons pour un opérateur différentiel à coefficients dans $K(z)$, $K$ un corps quelconque, la notion d'opérateur différentiel MOM en zéro. Dans la partie~\ref{210} nous montrerons, grâce au point 2 du théorème~\ref{practico}, que la série 210 de \cite{calabi} appartient à $\mathcal{L}(\mathcal{S})$, où $\mathcal{S}$ est un ensemble infini de nombres premiers. La pertinence de cet exemple provient du fait que les méthodes combinatoires et les techniques développées dans \cite[section 8]{Borisgfonct} ne parviennent pas encore à montrer que cette série appartient à $\mathcal{L}(\mathcal{S})$. La preuve du théorème~\ref{practico} repose sur le lemme ~\ref{ordre1}. Notre approche de la démonstration du lemme~\ref{ordre1} est fondée sur la théorie des équations différentielles $p$-adiques. Dans la partie~\ref{enonce}, nous démontrons le théorème~\ref{practico}, en admettant le lemme~\ref{ordre1}. Celui-ci est prouvé ultérieurement dans la partie~\ref{demordre1}. Également, le lecteur trouvera dans cette partie la définition de \emph{structure de Frobenius forte} (voir d\'efinition~\ref{def:SFF}). La partie~\ref{exemple} est consacrée à montrer que la série hypergéométrique $\mathfrak{f}_{2}(z)$ appartient à $\mathcal{L}^2(\mathcal{S})\setminus\mathcal{L}(\mathcal{S})$, quel que soit l'ensemble infini $\mathcal{S}$ de nombres premiers. Dans la partie~\ref{eqcalabi}, nous appliquons le  théorème~\ref{practico} \`a l'\'etude des solutions d'op\'erateurs diff\'erentiels de type Calabi-Yau. Enfin, la partie~\ref{independaalgebrique} est dédiée à l'étude de l'indépendance algébrique de certaines séries de $\mathcal{L}^2(\mathcal{S})$ et nous démontrerons le théorème~\ref{independece}. \medskip

\textbf{Remerciements.} L'auteur tient à remercier l'arbitre pour sa lecture attentive ainsi que pour ses divers commentaires qui ont clarifié l'exposition.

\section{Opérateurs différentiels}\label{opérateurs différentiels}
Dans cette partie nous allons rappeler quelques notions classiques des opérateurs différentiels, la lettre $K$ désigne un corps quelconque. Un opérateur différentiel à coefficients dans $K(z)$ est un objet de la forme
\begin{equation}\label{L}
L:=\frac{d^n}{dz^n}+a_{1}(z)\frac{d^{n-1}}{dz^{n-1}}+\cdots+a_{n-1}(z)\frac{d}{dz}+a_{n}(z),
\end{equation} 
où les $a_i(z)$ appartiennent à $K(z)$. \emph{L'ordre} de $L$ est $n$. Un point $\alpha$ dans la clôture algébrique de $K$ est un point \emph{singulier} de $L$ s'il existe $i\in\{1,\ldots,n\}$ tel que $\alpha$ est un pôle pour $a_i(z)$. Nous disons que $\alpha$ est un point \emph{singulier régulier} de $L$ si $\alpha$ est un point singulier de $L$ et, pour tout $i\in\{1,\ldots,n\}$, $\widetilde{a}_i(z):=(z-\alpha)^ia_i(z)$ n'a pas de pôle en $\alpha$. 
Soit $L_{\infty}$ l'opérateur différentiel obtenu après avoir appliqué le changement de variable $z\mapsto1/z$ à $L$. Nous disons que l'infini est un point \emph{singulier régulier} de $L$ si zéro est un point singulier régulier de $L_{\infty}$. 
Remarquons que si, pour chaque $i\in\{1,\ldots,n\}$, nous écrivons $a_i(z)=\frac{b_{i,1}(z)}{b_{i,2}(z)}$, où $b_{i,1}(z)$ et $b_{i,2}(z)$ appartiennent à $K[z]$ et sont premiers entre eux alors, l'infini est un point singulier régulier de $L$ si et seulement si, pour tout $i\in\{1,\ldots,n\}$, $deg(b_{i,1}(z))\leq deg(b_{i,2}(z))-i$. 

Nous dirons que $\alpha$ est une \emph{singularité à distance finie} de $L$ si $\alpha$ est un point singulier de $L$ et $\alpha$ n'est pas l'infini. Soit $\delta=z\frac{d}{dz}$ l'opérateur d'Euler. Nous réécrivons $z^nL$ en fonction de $\delta$ et obtenons $$L_{\delta}:=\delta^n+b_1(z)\delta^{n-1}+\cdots+b_{n-1}(z)\delta+b_n(z),$$
où les $b_i(z)$ appartiennent à $K(z)$. Comme nous le verrons dans la proposition suivante, si zéro est un point singulier régulier de $L$ alors pour tout $i\in\{1,\ldots,n\}$, $b_i(z)\in K[[z]]$. Ainsi, lorsque zéro est un point singulier régulier de $L$ nous définissons les \emph{exposants} en zéro de $L$ comme les racines du polynôme $$x^n+b_1(0)x^{n-1}+\cdots+b_{n-1}(0)x+b_n(0).$$ 
Ce polynôme est connu comme le polynôme \emph{indiciel} en zéro de $L$. Si l'infini est un point singulier régulier de $L$, nous définissons les \emph{exposants} en l'infini de $L$ comme les exposants en zéro de $L_{\infty}$.

\begin{remarque}\label{rem_g_n}
Soit $A(z)$ la matrice compagnon de $L$ et soit $B(z)$ la matrice compagnon de $L_{\delta}$. Alors,  $A(z)G_n=\frac{d}{dz}G_n+G_n\frac{1}{z}B(z)$, où \[
	G_n=\begin{pmatrix}
	1 & 0 & 0 &  \dots  & 0\\
	0 & \frac{1}{z} & 0 &  \dots & 0\\
	0 & \frac{c_{1,2}}{z^{2}} & \frac{1}{z^{2}} &  \dots & 0\\ 
	\vdots & \vdots & \vdots &\ddots  &\vdots\\
	0 & \frac{c_{1,j-1}}{z^{j-1}} &  \frac{c_{2,j-1}}{z^{j-1}} &\dots & \frac{1}{z^{j-1}}
	\end{pmatrix}\in {\rm GL}_n(\mathbb{Z}[1/z])
	\]
	est la matrice exprimant $\{1,\dots,\frac{d^{n-1}}{dz^{n-1}}\}$ en fonction de $\{1,\delta,\ldots,\delta^{n-1}\}$. Le lecteur trouvera une preuve de cela dans \cite[Lemme 5.3]{vmsff}.
\end{remarque}

\begin{prop}\label{Fuchs}
	Soient $K$ un corps quelconque et $L$ défini comme en \eqref{L}. Le point zéro est un point singulier régulier de $L$ si et seulement si, pour tout $i\in\{1,\ldots,n\}$, $b_i(z)$ appartient à $K[[z]]$.
\end{prop}
 Lorsque $K=\mathbb{C}$, ce résultat est le critère de Fuchs dont on trouve une démonstration dans \cite[corollaire~5.5]{Singer}.
Maintenant rappelons la définition d'opérateur fuchsien.
\begin{defi}[Opérateur fuchsien]
	Soient $K$ un corps quelconque et $L$ défini comme en \eqref{L}. L'opérateur $L$ est \emph{fuchsien} si l'infini est un point singulier régulier de $L$ et si toutes les singularités à distance finie de $L$ sont des points singuliers réguliers de $L$. Par abus de langage nous dirons que l'opérateur différentiel $L_{\delta}$ est fuchsien si, lorsque nous réécrivons l'opérateur $\frac{1}{z^n}L_{\delta}$ en fonction de $d/dz$, il est fuchsien.
\end{defi}

Dans la suite de l'article, nous utiliserons à plusieurs reprise la notion de MOM en zéro, notion algébrique définie ci-dessous qui généralise la notion de monodromie unipotente maximale lorsque $K=\mathbb{C}$.

\begin{defi}[Opérateur MOM en zéro]
	Soient $K$ un corps quelconque et $L$ défini comme en \eqref{L}. L'opérateur différentiel $L$ est MOM\footnote{L'abréviation MOM en anglais signifie maximal order multiplicity.} en zéro si: $L$ est fuchsien, zéro est un point singulier régulier de $L$ et tous les exposants en zéro de $L$ sont égaux à zéro.
\end{defi}

Soit $L\in\mathbb{C}(z)[d/dz]$  un opérateur fuchsien tel que zéro est un point singulier.  Il est bien connu que $L$ la matrice de monodromie locale en zéro est unipotente (MUM) si et seulement si tous les exposants en zéro de $L$ sont égaux à zéro. Ainsi, $L$ est MUM en zéro si et seulement si $L$ est MOM en zéro.

\section{Stratégie de la démonstration du théorème \ref{practico}}
Dans cette partie nous allons esquisser les idées principales de la démonstration du théorème~\ref{practico}. Cette démonstration repose essentiellement sur la \emph{théorie des équations différentielles sur le corps $\mathbb{F}_p$} et la \emph{théorie des équations différentielles $p$-adiques}. La démonstration de ce théorème a quatre étapes clés. A savoir :\medskip

--- Concernant la théorie des équations différentielles sur  $\mathbb{F}_p$, nous montrons que si $\mathcal{D}_p\in\mathbb{F}_p(z)[\delta]$ est MOM en zero et $g(z)\in 1+z\mathbb{F}_p[[z]]$ est solution de $\mathcal{D}_p$ alors $g(z)=b(z)(\Lambda_p(g))^p$, où $b(z)$ appartient à $\mathbb{F}_p(z)\cap\mathbb{F}_p[[z]]$ et a une hauteur inférieure ou égale à $p\eta\gamma-1$ avec $\eta$ l'ordre de $\mathcal{D}_p$ et $\gamma$ le nombre des points singuliers de $\mathcal{D}_p$. Cela est démontré par la proposition~\ref{falg}.  \medskip

--- Ce qui concerne la théorie des équations différentielles $p$-adiques, nous utilisons principalement le fait que l'existence de la structure de Frobenius forte implique que l'opérateur différentiel a une base de solutions dans le disque générique de rayon 1. En utilisant ce fait nous montrons que si $f(z)\in\mathcal{F}(\mathcal{S})$ annule un opérateur différentiel dans $\mathbb{Q}(z)[\delta]$ d'ordre $n$,  MOM en zéro et avec $r$ points singuliers alors, pour presque tout $p\in\mathcal{S}$ et tout entier $i\geq0$, la série $\Lambda^i_p(f_{\mid p})$ est solution d'un opérateur différentiel à coefficients dans $\mathbb{F}_p(z)$, d'ordre $n$, MOM en zéro et avec $r$ points singuliers au plus. Cela est démontré par la proposition~\ref{recurrence}. \medskip

--- Grâce aux propositions~\ref{falg} et \ref{recurrence} nous montrons que, pour tout couple $(m,i)\in\mathbb{Z}^2_{\geq0}$, $\Lambda^i_p(f_{\mid p})=A_{m,i,p}(z)\Lambda^{m+i}_p(f_{\mid p})(z^{p^m})$, où $A_{m,i,p}(z)\in\mathbb{F}_p(z)\cap\mathbb{F}_p[[z]]$ a une hauteur inférieure à $Cp^m$ avec $C=2nr$.  Cela est prouvé par le lemme~\ref{ordre1}.\medskip

--- Finalement,  comme $f(z)$ annule un opérateur différentiel muni d'une structure de Frobenius forte pour tout $p\in\mathcal{S}$, nous sommes en mesure de montrer  le fait crucial qu'il existe un entier $l_p>0$ tel que $\Lambda^{2l_p}_p(f_{\mid p}(z))=\Lambda_p^{l_p}(f_{\mid p}(z))$. Cela est prouvé par le lemme~\ref{lemm_egal_crucial}. \medskip

Ces quatre étapes entraînent le théorème~\ref{practico}.
 
\subsection{Structure de Frobenius forte et faible et disque générique}
Un module différentiel $A$ a une structure de Frobenius faible s'il existe un module différentiel $B$ et un entier $h>0$ tels que les modules différentiels $A$ et $B^{\phi^h}$ sont isomorphes, où $B^{\phi^h}$ est le module différentiel obtenu après avoir appliqué $h$-fois le Frobenius à $B$. Si $A$ et $A^{\phi^h}$ sont isomorphes pour certain $h>0$ alors on dit que $A$ a une structure de Frobenius forte. A cause de la conjecture de Bombieri--Dwork, la notion de structure de Frobenius forte est très liée à la notion d'une base de solutions dans le disque générique de rayon 1 car  il s'ensuit de cette conjecture qu'un opérateur différentiel a une structure de Frobenius forte pour presque tout nombre premier $p$ si et seulement si, pour presque tout nombre premier $p$,  l'opérateur différentiel a une base de solutions dans le disque générique de rayon 1. Cette conjecture est encore ouverte. Par contre, d'une part comme nous avons déjà dit, la structure de Frobenius forte implique que l'opérateur différentiel a une base de solutions dans le disque générique de rayon 1 et d'autre part, lorsque un opérateur différentiel a une base de solutions dans le disque générique de rayon 1, il est muni de la structure de Frobenius faible (voir lemme~\ref{10}).  La proposition~\ref{recurrence} est cruciale dans la démonstration du théorème~\ref{practico} et elle est prouvée dans la partie~\ref{subsec_preuve_prop_recurrence}. Sa preuve repose sur le lemme~\ref{Deuxième pas} et dans les hypothèses de ce lemme nous supposons que l'opérateur a une base de solution dans le disque générique de rayon 1 et l'existence de la structure de Frobenius forte n'est pas supposée.

 Le fait de supposer dans le lemme~\ref{Deuxième pas} cette hypothèse trouve sa motivation dans l'observation suivante : sous les conditions du théorème~\ref{practico}, $f(z)$ annule un opérateur différentiel $\mathcal{H}$ muni d'une structure de Frobenius forte pour tout $p\in\mathcal{S}$ ainsi qu'un opérateur différentiel $\mathcal{D}$ qui est MOM en zéro alors, si $\mathcal{M}$ est l'opérateur différentiel minimal pour $f(z)$, $\mathcal{H}=\mathcal{T}\mathcal{M}$ et $\mathcal{M}$ est aussi MOM en zéro. Cependant, $\mathcal{M}$ n'est pas forcement muni d'une structure de Frobenius forte pour tout $p\in\mathcal{S}$ mais, il est clair que $\mathcal{M}$ a une base de solutions dans le disque générique de rayon 1 pour tout $p\in\mathcal{S}$. Donc, le lemme~\ref{Deuxième pas} est appliqué à l'opérateur $\mathcal{M}$.

Ainsi, l'hypothèse de la structure de Frobenius est utilisée pour garantir que $\mathcal{M}$ a une base de solutions dans le disque générique de rayon 1 et pour établir l'égalité cruciale $\Lambda^{2l_p}_p(f_{\mid p}(z))=\Lambda_p^{l_p}(f_{\mid p}(z))$.

\section{Un premier exemple}\label{210}
 À présent nous allons utiliser la deuxième partie du théorème~\ref{practico}  pour montrer que la série 210 de \cite{calabi} appartient à $\mathcal{L}(\mathcal{P}\setminus\mathcal{J})$, où $\mathcal{J}$ est un ensemble fini de nombres premiers. Comme nous l'avons déjà mentionné, l'importance de cet exemple est donnée par le fait que les méthodes combinatoires et les techniques développées dans \cite[section 8]{Borisgfonct} ne semblent pas suffire à montrer que cette série appartient à $\mathcal{L}(\mathcal{S})$, où $\mathcal{S}$ est un ensemble infini de nombres premiers. La série en question est donnée par l'expression suivante $$f(z)=\sum_{j\geq0}\left(\binom{2j}{j}\left(\sum_{k=0}^{2j}(-1)^k\binom{2j}{k}^4\right)\right)z^j\in1+z\mathbb{Z}[[z]].$$
D'après le théorème~3.5 de \cite{sb}, la série $f(z)$ est la diagonale d'une fraction rationnelle, alors de \cite{picardfuchs}, $f(z)$ annule un opérateur différentiel $\mathcal{H}\in\mathbb{Q}(z)[d/dz]$ muni d'une structure de Frobenius forte pour presque tout nombre premier $p$. Comme la série fait partie de la liste donnée dans \cite{calabi}, elle annule un opérateur différentiel $\mathcal{D}\in\mathbb{Q}(z)[d/dz]$ qui est MOM en zéro. Nous allons voir que pour tout $p\in\mathcal{P}\setminus\{2\}$, $\Lambda_p(f(z))_{\mid p}=f_{\mid p}(z)$, c'est-à-dire que nous montrerons que pour tout $p\in\mathcal{P}\setminus\{2\}$ et tout entier positif $j$, $$\binom{2jp}{jp}\left(\sum_{k=0}^{2jp}(-1)^k\binom{2jp}{k}^4\right)\equiv\binom{2j}{j}\left(\sum_{k=0}^{2j}(-1)^k\binom{2j}{k}^4\right)\mod p.$$
Soit $k\in\{0,\ldots, 2jp\}$ alors $k=lp+s$, où $0\leq s\leq p-1$ et $0\leq l\leq 2j$. D'après le théorème de Lucas \cite[Section~2.1]{lucas}, on a que $$\binom{2jp}{k}^4=\binom{2jp}{lp+s}^4\equiv\binom{2j}{l}^4\binom{0}{s}^4\mod p.$$
En conséquence, si $s>0$ alors $\binom{2jp}{k}^4\equiv0\mod p$ et si $s=0$ alors $k=lp$ et $\binom{2jp}{lp}^4\equiv\binom{2j}{l}^4\mod p.$ Par conséquent, $$\sum_{k=0}^{2jp}(-1)^k\binom{2jp}{k}^4\equiv\sum_{l=0}^{2j}(-1)^{lp}\binom{2j}{l}^4=\sum_{l=0}^{2j}(-1)^l\binom{2j}{l}^4\mod p,$$
car $p\neq2$. 

Finalement, d'après le théorème de Lucas, $\binom{2jp}{jp}\equiv\binom{2j}{j}\mod p$ donc, $$\binom{2jp}{jp}\left(\sum_{k=0}^{2jp}(-1)^k\binom{2jp}{k}^4\right)\equiv\binom{2j}{j}\left(\sum_{k=0}^{2j}(-1)^k\binom{2j}{k}^4\right)\mod p,$$
pour tout nombre premier $p$ différent de 2 et tout entier positif $j$. Alors, on est en mesure d'appliquer le théorème~\ref{practico} et ainsi $f(z)\in\mathcal{L}(\mathcal{P}\setminus\mathcal{J})$, où $\mathcal{J}$ est un ensemble fini de nombres premiers. Nous remarquons qu'on peut appliquer le même raisonnement pour montrer que de nombreuses séries qui apparaissent dans \cite{calabi} appartiennent à $\mathcal{L}(\mathcal{S})$, où $\mathcal{P}\setminus\mathcal{S}$ est fini, voir par exemple la partie~\ref{eqcalabi}.

\section{Démonstration du théorème \ref{practico}}\label{enonce}

La preuve du théorème~\ref{practico} repose sur les lemmes~\ref{lemm_egal_crucial} et \ref{ordre1}. Nous allons énoncer les lemmes~\ref{lemm_egal_crucial} et \ref{ordre1} et nous passons tout de suite à la démonstration du théorème~\ref{practico}. La démonstration du lemme~\ref{lemm_egal_crucial} est faite à la fin de cette partie. Le lemme~\ref{ordre1} repose sur les propositions~\ref{falg} et \ref{recurrence} et sa démonstration fait le sujet de la partie~\ref{demordre1}. 

\begin{lem}\label{lemm_egal_crucial}
Supposons  que $f(z)=\sum_{n\geq0} a(n)z^n\in\mathbb{Z}_{(p)}[[z]]$ annule un opérateur différentiel muni d'une structure de Frobenius forte pour $p$. Alors, il existe un entier $l_p>0$ tel que $\Lambda^{2l_p}_p(f_{\mid p}(z))=\Lambda^{l_p}_p(f_{\mid p}(z))$. 
\end{lem}

	 
 
\begin{lem}\label{ordre1}
	Soit $\mathcal{S}$ un ensemble infini de nombres premier et supposons que $f(z)\in\mathcal{F}(\mathcal{S})$. Si $f(z)$ annule un opérateur différentiel $\mathcal{D}\in\mathbb{Q}(z)[d/dz]$ MOM en zéro alors, il existe $\mathcal{S}'\subset\mathcal{S}$ infini et une constante strictement positive $C$ tels que: l'ensemble $\mathcal{S}\setminus\mathcal{S}'$ est fini et, pour tout $p\in\mathcal{S}'$ et pour tout couple d'entiers $i,m$ supérieurs ou égaux à zéro, il existe une fraction rationnelle $A_{p,i,m}(z)\in\mathbb{F}_p(z)\cap\mathbb{F}_p[[z]]$ de hauteur inférieure ou égale à $Cp^m$ telle que $\Lambda^i_p(f_{\mid p}(z))=A_{p,i,m}(z)(\Lambda^{i+m}_p(f_{\mid p}(z))^{p^m}$.
\end{lem}

\begin{proof}[Démonstration du théorème~\ref{practico}]
	Par hypothèse on a que $f(z)=\sum_{n\geq0}a(n)z^n\in\mathcal{F}(\mathcal{S})$ et $f(z)$ annule un opérateur différentiel $\mathcal{D}\in\mathbb{Q}(z)[d/dz]$ MOM en zéro. Donc, d'après le lemme~\ref{ordre1}, il existe un ensemble infini $\mathcal{S}'\subset\mathcal{S}$ et une constante strictement positive $C$ tels que: $\mathcal{S}\setminus\mathcal{S}'$ est fini et pour tout $p\in\mathcal{S}'$ et tout couple d'entiers positifs $i,m$, il existe une fraction rationnelle $A_{p,i,m}(z)\in\mathbb{F}_p(z)\cap\mathbb{F}_p[[z]]$ de hauteur inférieure ou égale à $Cp^m$ telle que
	\begin{equation}\label{lambda}
	\Lambda^i_p(f_{\mid p}(z))=A_{p,i,m}(z)(\Lambda^{i+m}_p(f_{\mid p}(z)))^{p^m}.
	\end{equation}
	 \item 1-- Montrons que $f(z)\in\mathcal{L}^2(\mathcal{S}')$. Soit $p\in\mathcal{S}'$, comme $f(z)$ annule un opérateur différentiel muni d'une structure de Frobenius forte pour $p$ et $f(z)\in\mathbb{Z}_{(p)}[[z]]$ alors,  d'après le lemme~\ref{lemm_egal_crucial}, il existe un entier $l_p>0$ tel que $\Lambda^{2l_p}_p(f_{\mid p}(z))=\Lambda^{l_p}_p(f_{\mid p}(z))$. Donc, il suit de \eqref{lambda} qu'il existe $A_{p,l_p,l_p}(z)\in\mathbb{F}_p(z)\cap\mathbb{F}_p[[z]]$ de hauteur inférieur ou égale à $Cp^{l_p}$ telle que $$\Lambda^{l_p}_p(f_{\mid p}(z))=A_{p,l_p,l_p}(z)(\Lambda^{2l_p}_p(f_{\mid p}(z)))^{p^{l_p}}.$$ Comme $\Lambda^{2l_p}_p(f_{\mid p}(z))=\Lambda^{l_p}_p(f_{\mid p}(z))$ alors, $$\Lambda^{l_p}_p(f_{\mid p}(z))=A_{p,l_p,l_p}(z)(\Lambda^{l}_p(f_{\mid p}(z)))^{p^{l_p}} \text{ et }A_{p,l_p,l_p}(z)=\frac{\Lambda^{l_p}_p(f_{\mid p}(z))}{\Lambda^{l_p}_p(f_{\mid p}(z))^{p^l_p}}.$$ 
	
	 Il suit aussi de \eqref{lambda} qu'il existe $A_{p,0,l_p}(z)\in\mathbb{F}_p(z)\cap\mathbb{F}_p[[z]]$ de hauteur inférieure ou égale à $Cp^{l_p}$ telle que $$f_{\mid p}(z)=A_{p,0,l_p}(z)(\Lambda^{l_p}_p(f_{\mid p}(z)))^{p^{l_p}}\text{ et }\frac{1}{A_{p,0,l_p}(z)}=\frac{\Lambda^{l_p}_p(f_{\mid p}(z))^{p^{l_p}}}{f_{\mid p}(z)}.$$ Par conséquent, $\frac{A_{p,l_p,l_p}(z)}{A_{p,0,l_p}(z)}=\frac{\Lambda^{l_p}_p(f_{\mid p}(z))}{f_{\mid p}(z)}.$
	Finalement, soit $$A_p(z)=A_{p,0,l_p}(z)\left(\frac{A_{p,l_p,l_p}(z)}{A_{p,0,l_p}(z)}\right)^{p^{l_p}}.$$ Notons que le terme constant de $A_{p,0,l_p}(z)$ est 1 car le terme constant des séries $f_{\mid p}(z)$ et $\Lambda^{l_p}_p(f_{\mid p}(z))$ est 1. Donc, $\frac{1}{A_{p,0,l_p}(z)}\in\mathbb{F}_p(z)\cap\mathbb{F}_p[[z]]$. D'où, $A_p(z)\in\mathbb{F}_p(z)\cap\mathbb{F}_p[[z]]$. De plus, la hauteur de $A_p(z)$ est inférieure ou égale à $2Cp^{2l_p}$ et 
\begin{align*}
f_{\mid p}(z)=A_{p,0,l_p}(z)(\Lambda^{l_p}_p(f_{\mid p}(z)))^{p^{l_p}}=&A_{p,0,l_p}(z)\left(\frac{\Lambda^{l_p}_p(f_{\mid p}(z))}{f_{\mid p}(z)}\right)^{p^{l_p}}f_{\mid p}(z)^{p^{l_p}}\\
&=A_p(z)f_{\mid p}(z)^{p^{l_p}}.
\end{align*}		
Comme $\mathbb{F}_p$ est de caractéristique $p$ alors $f_{\mid p}(z)=A_p(z)f_{\mid p}(z^{p^{l_p}})$ . Il suit donc que $f(z)\in\mathcal{L}^2(\mathcal{S}')$.
	
	\medskip
	
	\item 2-- Montrons que $f(z)\in\mathcal{L}(S')$. Soit $p\in\mathcal{S}'$, par hypothèse on sait qu'il existe un entier $l_p>0$ tel que $\Lambda^{l_p}_p(f_{\mid p}(z))=f_{\mid p}(z)$. D'après \eqref{lambda}, il existe une fraction rationnelle $A_{p,0,l_p}(z)\in\mathbb{F}_p(z)\cap\mathbb{F}_p[[z]]$ de hauteur inférieure ou égale à $Cp^{l_p}$ telle que $f_{\mid p}(z)=A_{p,0,l_p}(z)(\Lambda^{l_p}_p(f_{\mid p}(z)))^{p^{l_p}}$. Ainsi, $f_{\mid p}(z)=A_{p,0,l_p}(z)f_{\mid p}(z)^{{p^{l_p}}}$. Comme $\mathbb{F}_p$ est de caractéristique $p$ alors  $f_{\mid p}(z)=A_{p,0,l_p}(z)f_{\mid p}(z^{{p^{l_p}}})$. D'où, $f(z)\in\mathcal{L}(\mathcal{S}')$.
\end{proof}

Nous finissons cette partie en démontrant le lemme~\ref{lemm_egal_crucial}.

\begin{proof}[Démonstration du lemme~\ref{lemm_egal_crucial}]
D'après le théorème~2.1 de \cite{vmsff} ou Christol \cite{Gillesalgebriques}, la série $f_{\mid p}(z)$ est algébrique sur $\mathbb{F}_p(z)$. Alors, d'après le théorème~1 de \cite{pautomata}, cela revient à dire que la suite $\{a(n)\bmod p\}_{n\geq0}$ est $p$-automatique. Mais, il découle de la proposition~3.3 de \cite[p. 107]{eilenberg} que la suite $\{a(n)\bmod p\}_{n\geq0}$ est $p$-automatique si et seulement si, l'ensemble $\left\{\sum_{n\geq0}(a(p^sn+d)\bmod p)z^n, s\geq0, 0\leq d<p\right\}$ est fini. En particulier, il existe un entier $a$  positif et un entier $b$ strictement positif tels que $\Lambda^{a}_p(f_{\mid p}(z))=\Lambda^{a+b}_p(f_{\mid p}(z))$. Soit $c$ un entier strictement positif tel que $cb>a$ et soit $l_p=cb$. En particulier, $\Lambda^a_p(f_{\mid p}(z))=\Lambda^{a+l_p}_p(f_{\mid p}(z))$.  Par conséquent,	 \begin{align*}
	 \Lambda^{l_p}_p(f_{\mid p}(z))=\Lambda^{l_p-a}_p(\Lambda^a_p(f_{\mid p}(z))=&\Lambda^{l_p-a}_p(\Lambda^{a+l_p}_p(f_{\mid p}(z)))=\Lambda^{2l_p}_p(f_{\mid p}(z)).
	 \end{align*}
\end{proof}

\section{Démonstration du lemme \ref{ordre1}}\label{demordre1}

La démonstration du lemme~\ref{ordre1} repose sur les propositions~\ref{falg} et \ref{recurrence} énoncées ci-dessous. Avant de les énoncer nous allons introduire la définition d'un opérateur différentiel $p$-unipotent. Pour chaque nombre premier $p$ nous notons $E_p$ le corps des éléments analytiques. Le lecteur trouvera par exemple dans \cite[section 3]{vmsff} la définition de ce corps. Nous soulignons que ce corps est muni de la norme ultramétrique de Gauss, nous notons $\vartheta_{E_p}$ l'anneau des éléments de $E_p$ tels que leurs normes sont inférieures ou égales à 1 et $\mathfrak{m}_p$ son unique idéal maximal. Le corps résiduel de $E_p$ est par définition $\vartheta_{E_p}/\mathfrak{m}_p$ et d'après le corollaire~1.3 de \cite{Gillesalgebriques}, $\vartheta_{E_p}/\mathfrak{m}_p$ est contenu dans $\overline{\mathbb{F}_p}((z))$, où $\overline{\mathbb{F}_p}$ est la clôture algébrique de $\mathbb{F}_p$. De plus, grâce au lemme~4.1 de \cite{vmsff} il existe un isomorphisme $\phi:\vartheta_{E_p}/\mathfrak{m}_p\rightarrow\overline{\mathbb{F}_p}(z)$. Il suit, d'après la définition de l'isomorphisme $\phi$, que tout élément de $\vartheta_{E_p}/\mathfrak{m}_p$ est une fraction rationnelle à coefficients dans un corps fini de caractéristique $p$. Ainsi, le corps $\vartheta_{E_p}/\mathfrak{m}_p$ est contenu dans $\overline{\mathbb{F}_p}(z)$. Étant donné un opérateur différentiel $\mathcal{D}$ à coefficients dans $\vartheta_{E_p}$, l'opérateur différentiel $\mathcal{D}_{p}$ à coefficients dans $\overline{\mathbb{F}_p}(z)$ est l'opérateur différentiel obtenu après avoir réduit chaque coefficient de $\mathcal{D}$ modulo l'idéal maximal $\mathfrak{m}_p$. Remarquons qu'en fait $\mathcal{D}_p$ est à coefficients dans $k(z)$, où $k$ est un corps fini de caractéristique $p$. Maintenant, nous introduisons la notion d'opérateur $p$-unipotent.

\begin{defi}\label{p-unipotent}[$p$-unipotent.]
Soit $\mathcal{D}$ un opérateur différentiel à coefficients dans $E_p$. On dit que $\mathcal{D}$ est $p$-unipotent si: \begin{enumerate}
	\item $\mathcal{D}$ est un opérateur différentiel à coefficients dans $\vartheta_{E_p}$. 
	\item L'opérateur différentiel $\mathcal{D}_{p}$ est non nul et MOM en zéro.
	\end{enumerate}
	\end{defi}

\begin{prop}\label{falg}
	Soient $p$ un nombre premier et $f(z)\in\mathbb{Z}_{(p)}[[z]]$. Soient $\mathcal{D}$ un opérateur $p$-unipotent d'ordre $n$ et $r$ le nombre de singularités à distance finie de $\mathcal{D}_p$ dans $\overline{\mathbb{F}_p}$. Si $f(z)$ est solution de $\mathcal{D}$ alors, il existe une fraction rationnelle $A_p(z)\in\mathbb{F}_p(z)$ de hauteur inférieure ou égale à $nrp-1$ telle que $f_{\mid p}(z)=A_p(z)(\Lambda_p(f_{\mid p}(z))^p$.
\end{prop} 

Sous les hypothèses de la proposition~\ref{falg}, si $f(0)=1$ alors $A_p(z)\in\mathbb{F}_p(z)\cap\mathbb{F}_p[[z]]$. 

\begin{defi}
	Étant donnée une série $f(z)\in\mathbb{Q}[[z]]$ qui annule un opérateur différentiel non nul à coefficients dans $\mathbb{Q}(z)$, nous notons $\mathcal{M}_f$ l'opérateur différentiel minimal de $f(z)$ à coefficients dans $\mathbb{Q}(z)$.
\end{defi}
L'opérateur différentiel $\mathcal{M}_f$ jouit des propriétés suivantes: d'abord il n'est pas nul et, deuxièmement un opérateur différentiel $\mathcal{L}$ à coefficients dans $\mathbb{Q}(z)$ est annulé par $f(z)$ si et seulement si $\mathcal{L}$ appartient à l'idéal $\mathbb{Q}(z)[d/dz]\mathcal{M}_f$.
\begin{prop}\label{recurrence}
	Soient $\mathcal{S}$ un ensemble infini de nombres premiers et $f(z)\in\mathcal{F}(\mathcal{S})$. Soient $n$ l'ordre de $\mathcal{M}_f$ et $r$ le nombre de singularité à distance finie de $\mathcal{M}_f$ dans $\overline{\mathbb{Q}}$. Si $f(z)$ annule un opérateur différentiel $\mathcal{D}\in\mathbb{Q}(z)[d/dz]$ MOM en zéro alors, il existe un ensemble $\mathcal{S}'\subset\mathcal{S}$ infini tel que: l'ensemble $\mathcal{S}\setminus\mathcal{S}'$ est fini et, pour tout $p\in\mathcal{S}'$ et tout entier positif $k$, la série $\Lambda^k_p(f(z))$ annule un opérateur $p$-unipotent $\mathcal{L}_k$ d'ordre $n$ tel que le nombre de singularités à distance finie de $\mathcal{L}_{k,p}$ dans $\overline{\mathbb{F}_p}$ est inférieur ou égal à $r$.
\end{prop}
Nous rappelons que l'opérateur différentiel $\mathcal{L}_{k,p}$ est l'opérateur obtenu après avoir réduit $\mathcal{L}_k$ modulo l'idéal maximal $\mathfrak{m}_p$.
\subsection{Démonstration du lemme \ref{ordre1}}\label{subsec_demos_lemme_ordre1}
Étant données les propositions~\ref{double} et \ref{recurrence} nous montrons le lemme~\ref{ordre1}.
\begin{proof}
	 Soient $n$ l'ordre de $\mathcal{M}_f$ et $r$ le nombre de singularités à distance finie de $\mathcal{M}_f$ dans $\overline{\mathbb{Q}}$. Par hypothèse la série $f(z)$ appartient à $\mathcal{F}(\mathcal{S})$ et annule un opérateur différentiel $\mathcal{D}\in\mathbb{Q}(z)[d/dz]$ MOM en zéro. Donc, on est en mesure d'appliquer la proposition~\ref{recurrence} et ainsi, il existe un ensemble $\mathcal{S}'\subset\mathcal{S}$ infini tel que: l'ensemble $\mathcal{S}\setminus\mathcal{S}'$ est fini et pour tout $p\in\mathcal{S}'$ et tout entier positif $k$, la série $\Lambda^k_p(f(z))$ annule un opérateur $p$-unipotent $\mathcal{L}_k$ d'ordre $n$ dont le nombre de singularités à distance finie de $\mathcal{L}_{k,p}$ dans $\overline{\mathbb{F}_p}$ est inférieur ou égal à $r$. Soient $p\in\mathcal{S}'$ et $i\in\mathbb{N}$, nous allons montrer par récurrence sur $m\in\mathbb{N}$ qu'il existe $A_{p,i,m}\in\mathbb{F}_p(z)\cap\mathbb{F}_p[[z]]$ de hauteur inférieure ou égale à $2nrp^m$ telle que
	\begin{equation}\label{lambdas}
	\Lambda^i_p(f_{\mid p}(z))=A_{p,i,m}(z)(\Lambda^{i+m}_p(f_{\mid p}(z)))^{p^m}.
	\end{equation}
	  Il est clair que \eqref{lambdas} est vrai pour $m=0$. Maintenant supposons que l'égalité \eqref{lambdas} est vraie pour $m$ et montrons qu'elle l'est aussi pour $m+1$. Donc, d'après notre hypothèse de récurrence il existe $A_{p,i,m}(z)\in\mathbb{F}_p(z)\cap\mathbb{F}_p[[z]]$ de hauteur inférieure ou égale à $2nrp^{m}$ telle que $\Lambda^i_p(f_{\mid p}(z))=A_{p,i,m}(z)(\Lambda^{i+m}_p(f_{\mid p}(z)))^{p^m}$. Comme $p\in\mathcal{S}'$ alors, d'après la proposition~\ref{recurrence}, la série $\Lambda^{i+m}_p(f(z))$ annule un opérateur $p$-unipotent $\mathcal{L}_{i+m}$ d'ordre $n$ dont le nombre de singularités à distance finie de $\mathcal{L}_{i+m,p}$ est inférieur ou égal à $r$. Ainsi, on est en mesure d'appliquer la proposition~\ref{falg} et par conséquent, il existe une fraction rationnelle $A_p(z)\in\mathbb{F}_p(z)\cap\mathbb{F}_p[[z]]$ de hauteur inférieure ou égale à $nrp-1$ telle que $\Lambda^{i+m}_p(f_{\mid p}(z))=A_p(z)(\Lambda^{i+m+1}_p(f_{\mid p}(z)))^p.$ D'où, $(\Lambda^{i+m}_p(f(z)))^{p^m}=A_p(z)^{p^m}(\Lambda^{i+m+1}_p(f_{\mid p}(z)))^{p^{m+1}}$. Il suit donc que, $$\Lambda^{i}_p(f_{\mid p}(z))=A_{p,i,m}(z)A_p(z)^{p^m}(\Lambda^{i+m+1}_p(f_{\mid p}(z))^{p^{m+1}}.$$
	  On pose $A_{p,i,m+1}(z)=A_{p,i,m}(z)A_{p}(z)^{p^m}$. Alors, $A_{p,i,m+1}(z)\in\mathbb{F}_p(z)\cap\mathbb{F}_p[[z]]$ car $A_{p,i,m}(z),A_{p}(z)\in\mathbb{F}_p(z)\cap\mathbb{F}_p[[z]]$. Montrons que la hauteur de $A_{p,i,m+1}(z)$ est inférieure ou égale à $2nrp^{m+1}$. Comme la hauteur de $A_p(z)$ est inférieure ou égale à $nrp-1$ alors la hauteur de $A_p(z)^{p^m}$ est inférieure ou égale à $nrp^{m+1}-p^m$ et ainsi, la hauteur de $A_{p,i,m+1}(z)$ est inférieure ou égale à $2nrp^m+ nrp^{m+1}-p^m$. Mais  $2nrp^m+ nrp^{m+1}-p^m=nrp^{m+1}+p^m(2nr-1)$ et $p^m(2nr-1)\leq nrp^{m+1}$ car $2nr-1\leq nrp$. Par conséquent, la hauteur de $A_{p,i,m+1}(z)$ est inférieure ou égale à $2nrp^{m+1}$.
	  Finalement, on pose $C=2nr$. Donc $C$ est une constante strictement positive indépendante de $p$. Ainsi nous avons montré que pour tout $p\in\mathcal{S}'$ et tout couple d'entiers positifs $i,m$, il existe une fraction rationnelle $A_{p,i,m}(z)\in\mathbb{F}_p(z)\cap\mathbb{F}_p[[z]]$ de hauteur inférieure ou égale à $Cp^{m}$ telle que $$\Lambda^i_p(f_{\mid p}(z))=A_{p,i,m}(z)(\Lambda^{i+m}_p(f_{\mid p}(z))^{p^m}.$$ 
\end{proof}	

\subsection{Démonstration de la proposition \ref{falg}}

Notre but dans cette partie est de montrer la proposition~\ref{falg}. Pour se faire nous aurons besoins de quelques résultats  concernant la théorie des opérateurs différentiels à coefficients dans un corps de caractéristique non nulle. Étant donné un opérateur différentiel $\mathcal{D}$ à coefficients dans $K(z)$, $K$ un corps quelconque, les ensembles $Ker(K(z),\mathcal{D})$ et $Ker(K((z)),\mathcal{D})$ désignent respectivement les solutions de $\mathcal{D}$ qui sont dans $K(z)$ et $K((z))$. Dans le cas où $K$ est un corps de caractéristique $p$, les deux ensembles sont respectivement $K(z^p)$, $K((z^p))$-espaces vectoriels.  Les lemmes~\ref{double1} et \ref{double} nous montrent que lorsque $K$ est un corps de caractéristique $p$ et $\mathcal{D}$ est MOM en zéro alors  $Ker(K((z)),\mathcal{D})$ a dimension égale à 1. Ces lemmes trouve leur motivation dans la remarque suivante.

\begin{remarque} Si $\mathcal{D}$ est un opérateur MOM en zéro à coefficients dans $\mathbb{Q}(z)$, il est bien connu que $Ker(\mathbb{Q}[[z]],\mathcal{D})$ a dimension égale à 1.
\end{remarque}
\begin{lem}\label{double1}
	Soient $\mathcal{D}$ un opérateur différentiel $p$-unipotent et $k$ un corps fini de caractéristique $p$ tel que $\mathcal{D}_p$ est à coefficients dans $k(z)$. Supposons $ker(k(z),\mathcal{D}_p)\neq0$. Alors $dim_{k(z^{p})}Ker(k(z),\mathcal{D}_p)=1$.
\end{lem}

\begin{proof}
	Par hypothèse il existe une fraction rationnelle $P(z)\in k(z)$ non nulle telle que $\mathcal{D}_{p}(P(z))=0$.
	Écrivons \[\mathcal{D}_{p}:=a_{0}(z)\frac{d^{n}}{dz^{n}}+a_{1}(z)\frac{d^{n-1}}{z^{n-1}}+\cdots+a_{n-1}(z)\frac{d}{dz}+a_{n}(z),\]
	o\`u les $a_{i}(z)$ appartiennent \`a $k[z].$ 
	Consid\'erons $\mathcal{D}_{\infty,p}$ l'opérateur différentiel \[\mathcal{D}_{\infty,p}:=b_{0}(z)\frac{d^{n}}{dz^{n}}+b_{1}\frac{d^{n-1}}{dz^{n-1}}+\cdots+b_{n-1}(z)\frac{d}{dz}+b_{n}(z),\]
	o\`u le vecteur $(b_{0}(z),\ldots, b_{n}(z))^t\in k[z]^{n+1}$ est égal à
	{\small
		\[
		\frac{(-1)^{n}}{z^{2n}}
		\begin{pmatrix}
		a_{0}(\frac{1}{z}),a_{1}(\frac{1}{z}),\ldots, a_{n-1}(\frac{1}{z}),a_{n}(\frac{1}{z})\\	
		\end{pmatrix}
		\begin{pmatrix}
		1 & 0 & 0 & 0 & \dots & 0 & 0\\
		0 & -z^{2} & 0 & 0 & \ldots & 0 & 0\\
		0 & 2z^{3} & z^{4} & 0 & 0 \ldots & 0 & 0\\
		\vdots & \vdots & \vdots & \vdots & \vdots\\
		0 & * & * & * & \dots & * & 0\\
		0 & * & * & * & \dots & * & (-z^{2})^{n}\\
		\end{pmatrix}
		.\]
	}
	 Cette matrice exprime le vecteur $(1,D,\ldots, D^{n})$ en fonction du vecteur $(1,\frac{d}{dz},\ldots,\frac{d^{n}}{z^{n-1}})$, où $D=-z^2\frac{d}{dz}$. 
	Comme $\mathcal{D}_p$ est singulier régulier en zéro et ses exposants en zéro sont tous égaux, alors l'opérateur différentiel $\mathcal{D}_{\infty,p}$ est singulier régulier en l'infini et ses exposants en l'infini sont tous égaux. De plus, $Ker(k(z),\mathcal{D}_{\infty,p})\neq0$ car $\mathcal{D}_{\infty,p}(P(\frac{1}{z}))=0$ et $P(\frac{1}{z})$ est non nulle. Montrons que $dim_{k(z^{p})}Ker(k(z),\mathcal{D}_{\infty,p})=1$. Pour ce faire on va démontrer que si $P_1(z),P_2(z)\in Ker(k(z),\mathcal{D}_{p,\infty})$ sont non nulles alors il existe $c(z^p)\in k(z^p)$ non nulle telle que $P_2(z)=c(z^p)P_1(z)$. Raisonnons par l'absurde et supposons que $P_1(z),P_2(z)$ sont linéairement indépendantes sur $k(z^p)$. Grâce au lemme~1.1 de \cite[Chap.~III]{Dworkgfunciones}, on a que $(\frac{d}{dz}P_1(z))P_2(z)-P_1(z)(\frac{d}{dz}P_2(z))\neq0$. De plus, il est clair qu'il existe $D_1(z)$ et $D_2(z)$ dans $k[z]$ non nuls tels que $D_1(z)^pP_1(z)$ et $D_2(z)^pP_2(z)$ appartiennent à $k[z]\cap Ker(k(z),\mathcal{D}_{\infty,p})$. Comme la caractéristique du corps $k$ est $p$ alors,
	\begin{multline*}
	\left(\frac{d}{dz}(D_1(z)^pP_1(z))\right)D_2(z)^pP_2(z)-D_1(z)^pP_1(z)\left(\frac{d}{dz}(D_2(z)^pP_2(z))\right)=\\
	D_1(z)^pD_2(z)^p\left[\left(\frac{d}{dz}P_1(z)\right)P_2(z)-P_1(z)\left(\frac{d}{dz}P_2(z)\right)\right]\neq0.
	\end{multline*}
	 Ainsi, d'après le lemme~1.1 de \cite[Chap~III]{Dworkgfunciones}, les polynômes $D_1(z)^pP_1(z)$ et $D_2(z)^pP_2(z)$ sont linéairement indépendants sur $k(z^p)$. Soit $\mathfrak{A}$ l'ensemble constitué des entiers strictement positifs qui sont de la forme $deg(R(z))+deg(S(z))$, où $R(z)$ et $S(z)$ appartiennent à $k[z]\cap Ker(k(z),\mathcal{D}_{\infty,p})$ et sont linéairement indépendants sur $k(z^p)$. Notons que $\mathfrak{A}$ est non vide car $$deg(D_1(z)^pP_1(z))+deg(D_2(z)^pP_2(z))$$ appartient à $\mathfrak{A}$. Par conséquent, il existe deux polynômes $Q_1(z)$ et $Q_2(z)$ linéairement indépendants sur $k(z^p)$ tels que, $Q_1(z)$ et  $Q_2(z)$ sont dans $k[z]\cap Ker(k(z),\mathcal{D}_{\infty,p})$ et $deg(Q_1(z))+deg(Q_2(z))$ est l'élément le plus petit de $\mathfrak{A}$. Démontrons maintenant en utilisant les arguments donnés par Honda dans \cite[proposition~5.1]{honda} que
	\begin{equation}\label{1}
	deg(Q_1(z))\neq deg(Q_2(z))\mod p.
	\end{equation}
	Raisonnons par l'absurde et supposons que $deg(Q_1(z)-deg(Q_2(z))=vp\geq0$. Donc, il existe $c\in k$ non nul tel que $deg(Q_1(z)-cz^{vp}Q_2(z))<deg(Q_1(z))$. Soit $R_1(z)=Q_1(z)-cz^{vp}Q_2(z)$.  Les polynômes $R_1(z)$ et $Q_2(z)$ sont linéairement indépendants sur $k(z^p)$ car $Q_1(z)$ et $Q_2(z)$ le sont aussi. Notons que $R_1(z)$ est solution de $\mathcal{D}_{\infty,p}$ car $k$ est de caractéristique $p$. Alors, $deg(R_1(z))+deg(Q_2(z))$ appartient à $\mathfrak{A}$. Mais, $deg(R_1(z))+deg(Q_2(z))<deg(Q_1(z))+deg(Q_2(z)).$ Cela contredit le fait que $deg(Q_1(z))+deg(Q_2(z))$ est l'élément le plus petit de $\mathfrak{A}$. Donc, $deg(Q_1(z))\neq deg(Q_2(z))\mod p$. Maintenant, d'après la proposition~2.2 de \cite{honda}, on a que $-deg(Q_1(z))\mod p$ et $-deg(Q_2(z))\mod p$ sont exposants en l'infini de $\mathcal{D}_{\infty,p}$. Mais, nous savons que les exposants en l'infini de $\mathcal{D}_{\infty,p}$ sont tous égaux. Alors, 
	\begin{equation}\label{-1}
	deg(Q_1(z))=deg(Q_2(z))\mod p.
	\end{equation}
	Les égalités \eqref{1} et \eqref{-1} se contredisent et par conséquent, $P_1(z)$ et $P_2(z)$ sont linéairement dépendantes sur $k(z^p)$. D'où  il existe $c(z^p)\in k(z^p)$ non nulle telle que $P_2(z)=c(z^p)P_1(z)$. Ainsi, $dim_{k(z^{p})}Ker(k(z),\mathcal{D}_{\infty,p})=1$. Montrons que cette égalité implique que, $dim_{k(z^{p})}Ker(k(z),\mathcal{D}_{p})=1.$
	En effet, soit $B(z)\in Ker(k(z),\mathcal{D}_{p})$ non nulle, alors $B(1/z)\in Ker(k(z),\mathcal{D}_{\infty,p})$, d'o\`u il existe $c(z^{p})\in k(z^{p})$ non nulle telle que $B(1/z)=c(z^{p})P(1/z)$. Par cons\'equent, $B(z)=c(1/z^{p})P(z)$ et ainsi $dim_{k(z^{p})}Ker(k(z),\mathcal{D}_{p})=1.$
\end{proof}

\begin{lem}\label{double}
	Soient $\mathcal{D}$ un opérateur différentiel $p$-unipotent et $k$ un corps fini de caractéristique $p$ tel que $$\mathcal{D}_{p}:=a_{0}(z)\delta^n+a_{1}(z)\delta^{n-1}+\cdots+a_{n-1}(z)\delta+a_{n}(z)\in k[z][\delta],$$ où, pour tout $i\in\{1,\ldots,n\}$, le degré de $a_i(z)$ est inférieur ou égal à $d$. Si $f(z)\in k[[z]]$ est une solution non nulle de $\mathcal{D}_p$ alors il existe un polynôme non nul $P(z)\in k[z]$ de degré inférieur ou égal à $pd-1$ tel que $f(z)=P(z)c(z^p)$ où $c(z)\in k((z))$. 
\end{lem}

\begin{proof}
Comme $\mathcal{D}_p$ est MOM en zéro alors, il existe des polynômes $A_0(z),\ldots, A_{d-1}(z)\in k[z]$ de degré inférieur ou égal à $n$ tels que la série $\sum_{m\geq0}b_mz^m$ est une solution de $\mathcal{D}_p$ si et seulement si, la suite $(b_m)_{m\geq0}$ vérifie la relation de récurrence 
\begin{equation}\label{eq_recurrence}
m^nb_m=A_{d-1}(m)b_{m-1}+A_{d-2}(m)b_{m-2}+\cdots+A_0(m)f_{m-d}.
\end{equation}
 
Par hypothèse, il existe une série non nulle $f(z)=\sum_{m\geq0}f_mz^m$ telle que $f(z)\in ker(k[[z]],\mathcal{D}_p)$. Si $f_0=0$ alors, d'après la relation de recurrence~\eqref{eq_recurrence}, $f_1=\cdots=f_{p-1}=0$. Ainsi, $f(z)$ est divisible par $z^p$ et la série  $f(z)/z^p$ est une solution de $\mathcal{D}_p$, laquelle est dénotée à nouveau par $f(z)$. On fait tel processus jusqu'à obtenir $f(0)\neq0$. Donc, sans perte de généralité on suppose $f_0\neq0$. Pour chaque $i\in\{0,1,\ldots,d\}$, nous considérons le vecteur $v_i=(f_{ip}, f_{ip+1},\ldots, f_{ip+d-1})\in k^{d}$. Supposons que, pour tout $i\in\{1,\ldots, d\}$, $v_i\neq0$. Par conséquent, il existe $\alpha_0,\alpha_1,\ldots,\alpha_d\in k$, non tous nuls, tels que
\begin{equation}\label{eq_alg_lineal}
\alpha_d v_0+\alpha_{d-1}v_1+\cdots+\alpha_{0}v_d=0.
\end{equation}
Considérons $g(z)=(\alpha_j+\alpha_{j+1}z^p+\cdots+\alpha_{d}z^{p(d-j)})f(z)$. Alors, $g(z)$ est une solution non nulle de $\mathcal{D}_p$. Ecrivons $g(z)=\sum_{m\geq0}g_mz^m$. L'égalité~\eqref{eq_alg_lineal} nous donne que $g_{pd}=g_{pd+1}=\cdots=g_{pd+d-1}=0$. De plus, pour tout entier $s\geq0$, $$s^ng_s=A_{d-1}(s)g_{s-1}+A_{d-2}(s)g_{s-2}+\cdots+A_0(s)g_{s-d}.$$ Alors, il découle que $g_{pd+d}=g_{pd+d+1}=\cdots=g_{p(d+1)-1}=0$. Ainsi, $g=P(z)+z^{p(d+1)}h(z)$, où $P(z)=\sum_{m=0}^{pd-1}P_mz^m$ avec $P_m=g_m$ pour $0\leq m<pd$ et $P_m=0$ pour $m\geq pd$ et $h(z)=\sum_{m\geq0}g_{m+p(d+1)}z^m$. Comme $g_{pd}=g_{pd+1}=\cdots=g_{pd+d-1}=0$ alors la suite $(P_m)_{m\geq0}$ vérifie \eqref{eq_recurrence}. D'où, $P(z)$ est une solution de $\mathcal{D}_p$.  Ainsi, la série $h(z)$ est une solution de $\mathcal{D}_p$. De plus, $$f(z)=(P(z)+z^{p(d+1)}h(z))c_0(z^p),$$
avec $c_0(z^p)=(\alpha_0+\alpha_1z^p+\cdots+\alpha_{d}z^{pd})^{-1}.$

Supposons maintenant qu'il existe $i\in\{1,\ldots, d\}$ tel que $v_i=0$. Comme $(f_m)_{m\geq0}$ vérifie la récurrence~\eqref{eq_recurrence}, il découle que $f_{ip}=f_{ip+1}=\cdots=f_{p(i+1)-1}=0$. Ainsi, $f(z)=P(z)+z^{p(i+1)}h(z)$, où $P(z)=\sum_{m=0}^{ip-1}f_mz^m$ et $h(z)=\sum_{m\geq0}f_{m+p(i+1)}z^m$. Comme $f_{ip}=f_{ip+1}=\cdots=f_{ip+d-1}=0$ alors la suite $(P_m)_{m\geq0}$ avec $P_m=f_m$ pour $0\leq m<ip$ et $P_m=0$ pour $m\geq ip$ vérifie la récurrence~\eqref{eq_recurrence}. Donc, $P(z)$ est une solution de $\mathcal{D}_p$ et ainsi, $h(z)$ est aussi une solution de $\mathcal{D}_p$. 
Par conséquent, l'argument précédent montre que $$f(z)=(P_0(z)+z^{p(d+1)}h_1(z))c_0(z^p),$$ où $P_0(z)$ est un polynôme de degré inférieur ou égal à $pd-1$ qui annule $\mathcal{D}_p$, $h_1(z)$ est une série qui annule aussi $\mathcal{D}_p$ et $c_0(z)\in k((z))$. 

Comme $h_1(z)$ est solution de $\mathcal{D}_p$, d'après l'argument précédent, on en déduit que $h_1(z)=(P_1(z)+z^{p(d+1)}h_2(z))c_1(z^p)$, où $P_1(z)$ est un polynôme de degré inférieur ou égal à $pd-1$ qui annule $\mathcal{D}_p$, $h_2(z)$ est une série qui annule $\mathcal{D}_p$ et $c_1(z)\in k((z))$. Ainsi, par récurrence, on a $$f(z)=\sum_{i\geq0}z^{ip(d+1)}P_i(z)d_i(z^p),$$ où $P_i(z)$ est un polynôme de degré inférieur ou égal à $pd-1$ et $d_i(z)\in k((z))$. Il existe $i\geq0$ tel que $P_i(z)$ est non nul car $f(z)$ est non nulle. Sans perte de généralité, supposons $P_0(z)$ non nul. Donc, le lemme~\ref{double1}, entraîne que, pour tout $i\geq0$, il existe $r_i(z)\in k(z^p)$ tel que $P_i(z)=P_0(z)r_i(z^p)$. Donc, il existe $c(z)\in k((z^p))$ telle que $f(z)=P_0(z)c(z^p).$ 
\end{proof}

Nous sommes prêts maintenant à démontrer la proposition~\ref{falg}.

\begin{proof}[Démonstration de la proposition \ref{falg}]
	  Notons que si la série $f_{\mid p}(z)$ est nulle alors la proposition~\ref{falg} découle tout de suite. Supposons que $f_{\mid p}(z)$ est non nulle. Comme $\mathcal{D}$ est $p$-unipotent alors, par définition, $\mathcal{D}_p$ est fuchsien. Ainsi, $$\mathcal{D}_{p}:=a_{0}(z)\delta^n+a_{1}(z)\delta^{n-1}+\cdots+a_{n-1}(z)\delta+a_{n}(z)\in k[z][\delta],$$ où le degré de $a_i(z)$ est inférieur ou égal à $nr$ pour $0\leq i\leq n$. De plus, $\mathcal{D}_p$ est MOM en zéro. Alors, le lemme~\ref{double} nous garanti l'existence d'un polynôme $P(z)$ non nul de degré inférieur ou égal à $pnr-1$ tel que  $f_{\mid p}(z)=P(z)c(z^p)$. Notons que $\Lambda_p(P(z)c(z^p))=\Lambda_p(P(z))c(z)$ et ainsi, $\Lambda_p(f_{\mid p}(z))=\Lambda_p(P(z))c(z)$ et comme $f_{\mid p}(z)\in\mathbb{F}_p[z]$ alors, $\frac{\Lambda_p(f_{\mid p}(z))^p}{\Lambda_p(P(z^p))}=c(z^p)$. Par conséquent, $f_{\mid p}(z)=\frac{P(z)}{\Lambda_p(P(z^p))}\Lambda_p((f_{\mid p}(z))^p$. Finalement, comme $deg(\Lambda_p(P(z^p)))\leq deg(P(z))\leq nrp-1$ alors la hauteur de $\frac{P(z)}{\Lambda_p(P(z^p))}$ est inférieure ou égale à $nrp-1$. 
\end{proof} 

\subsection{Démonstration de la proposition \ref{recurrence}}\label{subsec_preuve_prop_recurrence}
Le but de cette partie est de démontrer la proposition~\ref{recurrence}. La démonstration que nous faisons de cette proposition repose sur la théorie des équations différentielles $p$-adiques.  

\subsubsection{Équations différentielles $p$-adiques} La démonstration de la proposition~\ref{recurrence} s'appuie essentiellement dans le lemme~\ref{Deuxième pas}. Avant de l'énoncer nous commençons avec la définition des ensembles $E_{0,p}$, $\mathscr{M}_p$, des disques singuliers réguliers et de la propriété $(\textbf{P})_{p,r,n}$. Nous rappelons que $\mathbb{C}_p$ est le complété de la clôture algébrique de $\mathbb{Q}_p$. Nous commençons par rappeler la notion de structure de Frobenius forte.

\begin{defi}[Structure de Frobenius forte]\label{def:SFF}
	Soient $\mathcal{H}\in\mathbb{Q}(z)[\delta]$, $A(z)$ la matrice compagnon de $\mathcal{H}$ et $p$ un nombre premier. On dit que l'opérateur différentiel $\mathcal{H}$ est muni d'une structure de Frobenius forte pour $p$ de période $h$ s'il existe une matrice inversible $H(z)$ à coefficients dans $E_p$ de même taille que $A$ et un entier $h>0$ tels que $$\delta H(z)=A(z)H(z)-p^hH(z)A(z^{p^h}).$$
\end{defi}

\begin{defi}
$E_{0,p}$ est l'ensemble des éléments analytiques $E_p$ qui n'ont pas de pôle dans le disque ouvert $D(0,1)$.		
\end{defi}
En particulier, l'ensemble $E_{0,p}$ est un anneau contenu dans $\mathbb{C}_p[[z]]$. La définition suivante a été introduite par Christol dans \cite[Chap~4]{christolpadique} dans un contexte plus large.
\begin{defi}\label{R}Soit $p$ un nombre premier, nous notons $\mathscr{M}_p$ l'ensemble des matrices carrées $L$ à coefficients dans $E_p$ qui satisfont aux conditions suivantes:	
	\begin{enumerate}[label=\roman*)]
		\item $L$ est une matrice dont les coefficients appartiennent à $E_{0,p}$.
		\item Les valeurs propres de la matrice $L(0)$ sont toutes égales à zéro. 
		\item Il existe une matrice inversible $U_t(z)$, dont les coefficients sont analytiques dans le disque générique ouvert $D(t,1)$, telle que $\delta(U)=LU$, où $\delta=z\frac{d}{dz}$ et $t$ est un élément transcendante sur $\mathbb{C}_p$ indépendant de $z$.
	\end{enumerate}
\end{defi}

\begin{remarque}\label{rem_motiv_m_p}
Notre motivation pour introduire l'ensemble $\mathscr{M}_p$ est donnée par le fait suivant.  Sous les hypothèses du théorème~\ref{practico}, on a que $f(z)$ annule un opérateur différentiel $\mathcal{H}$ muni d'une structure de Frobenius forte pour tout $p\in\mathcal{S}$ ainsi qu'un opérateur différentiel $\mathcal{D}$ qui est MOM en zéro. Soient $A$ la matrice compagnon de $\mathcal{H}$ et $B$ la matrice compagnon de $\mathcal{D}$. Alors, $A$ et $B$ n'appartiennent pas nécessairement à $\mathscr{M}_p$ car  les exposants en zéro de $\mathcal{H}$ ne sont pas forcement tous égaux à zéro et on ne sait pas si $\mathcal{D}$ a une base  de solutions dans le disque générique de rayon 1. Cependant, notre observation est que, pour presque tout $p\in\mathcal{S}$, la matrice $M_{f,\delta}$ est dans $\mathscr{M}_p$, où $M_{f,\delta}$ est la matrice compagnon de l'opérateur différentiel obtenu après avoir réécrit $z^n\mathcal{M}_f$ en fonction de $\delta$. Cela est prouvé dans le lemme suivant.

\begin{lem}\label{lem_minimal}
Soient $\mathcal{S}$ un ensemble infini de nombres premiers. Si $f(z)\in\mathcal{F}(\mathcal{S})$ annule un opérateur $\mathcal{D}\in\mathbb{Q}(z)[\delta]$ MOM en zéro alors, pour presque tout $p\in\mathcal{S}$, $M_{f,\delta}$ appartient à $\mathscr{M}_p$.
\end{lem}

 \begin{proof}
Soit $\mathcal{S}'$ l'ensemble de nombres premiers dans $\mathcal{S}$ tel que $\mathcal{M}_{f,\delta}$ est à coefficients dans $E_{0,p}$. Comme $\mathcal{M}_{f,\delta}$ est à coefficients dans $\mathbb{Q}(z)$ alors  $\mathcal{S}\setminus\mathcal{S}'$ est fini. De plus, l'égalité $\mathcal{D}=\mathcal{P}\mathcal{M}_{f,\delta}$ entraîne que $P_{\mathcal{D}}=P_{\mathcal{P}}P_{\mathcal{M}_{f,\delta}},$ où $P_{\mathcal{D}}$ est le polynôme indiciel de $\mathcal{D}$ et $\mathcal{P}_{\mathcal{M}_{f,\delta}}$ est le polynôme indiciel de $\mathcal{M}_{f,\delta}$. Par définition de MOM en zéro, toutes les racines de $P_{\mathcal{D}}$ sont égales à zéro. Donc, les racines de $\mathcal{P}_{\mathcal{M}_{f,\delta}}$ sont égales à zéro. Par conséquent,  les valuers propres de $M_{f,\delta}(0)$ sont toutes égales à zéro. Soit $\mathcal{H}$ muni d'une structure de Frobenius forte pour $p\in\mathcal{S}$. Alors, les propositions~4.1.2, 4.6.4 et 4.7.2 de \cite{Gillesmoduldiff} montrent qu'il existe une matrice inversible $W_t(z)$ dont les coefficients sont analytiques dans le disque générique ouvert $D(t,1)$ telle que $\delta W_t=AW_t$, où $A$ est la matrice compagnon de $\mathcal{H}$. Donc, de l'égalité $\mathcal{H}=\mathcal{T}\mathcal{M}_{f,\delta}$ découle qu'il existe une matrice inversible $U_t(z)$ dont les coefficients sont analytiques dans le disque générique ouvert $D(t,1)$ telle que $\delta U_t=M_{f,\delta}U_t$.
\end{proof}

\end{remarque}

Nous rappelons les notions de \emph{pôle} pour un élément de $E_p$ et de \emph{disque singulier régulier} pour une matrice à coefficients dans $E_p$. Pour chaque $\alpha\in\overline{\mathbb{F}_p}$, l'anneau $E_p^{\alpha}$ est la fermeture dans le cors $E_p$ de l'anneau des fractions rationnelles de $\mathbb{C}_p(z)$ dont tous les pôles appartiennent à $\vartheta_{\mathbb{C}_p}$ et ont $\alpha$ pour image dans le corps $\overline{\mathbb{F}_p}$ et qui sont nulles à l'infini. D'une manière analogue l'anneau $E^{\infty}_p$ est la fermeture dans le corps $E_p$ de l'ensembles des fractions rationnelles de $\mathbb{C}_p(z)$ qui n'ont pas de pôle dans $\vartheta_{\mathbb{C}_p}$. Donc, grâce au théorème de Mittag-Leffler, voir par exemple théorème~2.1.6 de \cite{Gillesmoduldiff}, à tout élément $a(z)\in E_p$ on peut associer de manière unique des éléments analytiques $a_{\alpha}(z)\in E^{\alpha}_p$, pour $\alpha$ parcourant l'ensemble $\overline{\mathbb{F}_p}\cup\{\infty\}$, tels que la famille $\{a_{\alpha}(z)\}$ tende vers zéro selon le filtre des complémentaires des parties finies de $\overline{\mathbb{F}_p}\cup\{\infty\}$ et qui vérifient $$a(z)=\sum\limits_{\alpha\in\overline{\mathbb{F}_p}\cup\{\infty\}}a_{\alpha}(z),\quad |a(z)|=sup|a_{\alpha}(z)|.$$  Soient $\gamma\in\vartheta_{\mathbb{C}_p}\cup\{\infty\}$ et $a(z)\in E_p$. Le point $\gamma$ est un \emph{pôle} de $a(z)$ si dans l'écriture donnée par le théorème de Mittag-Leffler l'élément qui correspond à l'anneau $E^{\overline{\gamma}}_p$ est non nul. 

Pour $\gamma\in\vartheta_{\mathbb{C}_p}$, nous notons $D_{\gamma}$ le disque ouvert de centre $\gamma$ et de rayon 1 et $D_{\infty}$ est l'ensemble des éléments de $\mathbb{C}_p$ qui ont une norme supérieure à 1. Remarquons que $D_{\gamma}=D_{\beta}$ si et seulement si $|\gamma-\beta|_p<1$. Nous avons aussi que $|\gamma-\beta|_p=1$ si, et seulement si $D_{\gamma}\cap D_{\beta}=\emptyset$. Soit $A\in M_n(E_p)$. La matrice $A$ est \emph{singulière régulière} dans le disque $D_{\gamma}$, s'il existe une matrice $A_{\gamma}$ telle que, les matrices $A$ et $A_{\gamma}$ sont $E_p$-équivalentes\footnote{Deux matrices $A,B\in M_n(E_p)$ sont $E_p$-équivalentes s'il existe $H\in Gl_n(E_p)$ telle que $\frac{d}{dz}H=AH-HB$. La relation de $E_p$- équivalence est une relation d'équivalence.} et il existe $\beta_{\gamma}\in D_{\gamma}$ tel que la matrice $(z-\beta_{\gamma})A_{\gamma}$ n'a pas de pôle dans le disque $D_{\gamma}$. La matrice $A$ est \emph{singulière régulière} dans le disque $D_{\infty}$, s'il existe une matrice $A_{\infty}$ telle que les matrice $A$ et $A_{\infty}$ sont $E_p$-équivalentes et la matrice $zA_{\infty}$ n'a pas de pôle dans le disque $D_{\infty}$.  
\begin{remarque}\label{trans}
	Soient $A, B\in M_n(E_p)$. Supposons que $A$ et $B$ sont $E_p$-équivalentes. La matrice $A$ est singulière régulière dans le disque $D_{\gamma}$ si et seulement si la matrice $B$ est singulière régulière dans le disque $D_{\gamma}$. En effet, si $A$ est singulière régulière dans le disque $D_{\gamma}$ alors il existe une matrice $A_{\gamma}$ telle que $A$ et $A_{\gamma}$ sont $E_p$-équivalentes et il existe $\beta_{\gamma}\in D_{\gamma}$ tel que $(z-\beta_{\gamma})A_{\gamma}$ n'a pas de pôle dans le disque $D_{\gamma}$. Comme $A$ et $B$ sont $E_p$-équivalentes alors par transitivité  les matrice $B$ et $A_{\gamma}$ sont $E_p$-équivalentes. Par conséquent, la matrice $B$ est singulière régulière dans le disque $D_{\gamma}$.  Maintenant, si $D_{\gamma}$ est un disque singulier régulier de $B$ alors, en appliquant le même argument précédent on obtient que la matrice $A$ est singulière régulière dans le disque $D_{\gamma}$.  
\end{remarque}

\begin{lem}\label{singularites}
	Soient $\mathcal{L}\in\vartheta_{E_p}[d/dz]$ un opérateur différentiel unitaire et $L(z)$ sa matrice compagnon. Si $L(z)$ a exactement $r$ disques singuliers à distance finie alors le nombre de singularités de l'opérateur $\mathcal{L}_{p}$ dans $\overline{\mathbb{F}_p}$ est inférieur ou égal à $r$.
\end{lem}
 
\begin{proof}
	Soient $D_{\gamma_1},\ldots, D_{\gamma_r}$ les disques singuliers de $L(z)$. Par hypothèse ils sont tous différents, d'où pour $i\neq j$ on a que $|\gamma_i-\gamma_j|=1$, ce qui revient à dire que $\overline{\gamma_i}\neq\overline{\gamma_j}$. Montrons que les singularités de $\mathcal{L}_p$ appartiennent à $\{\overline{\gamma_1},\ldots,\overline{\gamma_r}\}$. Soit $A_i(z)$ un des coefficients de $\mathcal{L}$ alors les pôles de $A_i(z)$ sont dans $D_{\gamma_1}\cup\ldots\cup D_{\gamma_r}$. Par conséquent, d'après le théorème de Mitagg-Leffler, les pôles de $\overline{A_i(z)}$ appartiennent à $\{\overline{\gamma_1},\ldots,\overline{\gamma_r}\}$. D'où les pôles de $\mathcal{L}_p$ sont dans $\{\overline{\gamma_1},\ldots,\overline{\gamma_r}\}$.
\end{proof}

Nous rappelons que pour un opérateur unitaire $\mathcal{L}\in E_{p}[d/dz]$ d'ordre $n$ nous dénotons par $\mathcal{L}_{\delta}\in E_{p}[\delta]$ l'opérateur obtenu après avoir réécrit $z^n\mathcal{L}$ en fonction de $\delta$. Nous notons par $L(z)$ la matrice compagnon de $\mathcal{L}$ et par $L_{\delta}(z)$ la matrice compagnon de $\mathcal{L}_{\delta}$.
 \begin{defi}\label{propriedadP}
 	Soit $f(z)\in 1+z\mathbb{C}_p[[z]]$. Nous disons que $f(z)$ satisfait la propriété $(\textbf{P})_{p,r,n}$ s'il existe un opérateur différentiel unitaire $\mathcal{L}\in E_{p}[d/dz]$, tel que:
 	\begin{enumerate}
 		\item L'opérateur $\mathcal{L}$ est annulé par $f(z)$.
 		\item L'ordre de $\mathcal{L}$ est $n$.
 		\item La matrice $L_{\delta}(z)$ appartient à $\mathscr{M}_p$.
 		\item L'opérateur $\mathcal{L}$ est $p$-unipotent.
 		\item Tous les disques singuliers de la matrice $L(z)$ sont singuliers réguliers.
 		\item La matrice $L(z)$ a exactement $r$ disques singuliers réguliers \emph{à distance finie}.\footnote{Le disque ouvert de centre $\gamma\in\mathbb{C}_p$ et de rayon 1 est à distance finie si la norme de $\gamma$ est inférieure ou égale à 1.}	 		
 	\end{enumerate}
 \end{defi}

\begin{lem}\label{Deuxième pas}
	Soit $f(z)\in 1+z\mathbb{C}_p[[z]]$. Si $f(z)$ vérifie la propriété $(\textbf{P})_{p,r,n}$ alors $\Lambda_p(f(z))$ vérifie aussi la propriété $(\textbf{P})_{p,r,n}$.
\end{lem}
 
En particulier, si $f(z)\in 1+z\mathbb{C}_p[[z]]$ vérifie la propriété $(\textbf{P})_{p,r,n}$ alors, pour tout entier $k$ positif, la série $\Lambda^k_p(f(z))$ vérifie la propriété $(\textbf{P})_{p,r,n}$. Le lemme~\ref{Deuxième pas} sera prouvé dans la partie~\ref{sec_dem_deuxieme_pas} et dans la partie suivante nous allons voir que, sous les hypothèses du théorème~\ref{practico},  pour presque tout $p\in\mathcal{S}$, l'opérateur $\mathcal{M}_f$ satisfait aux conditions (1)-(6) de la définition~\ref{propriedadP}.

\subsubsection{Construction de l'ensemble $B_\mathcal{S}$}\label{construction}
Soit $\mathcal{S}$ un ensemble infini de nombres premiers et soient $f(z)\in\mathcal{F}(\mathcal{S})$,  $n$ l'ordre de $\mathcal{M}_f$ et $r$ le nombre de singularités à distance finie de $\mathcal{M}_f$ dans $\overline{\mathbb{Q}}$.  Dans cette partie nous allons construire un ensemble $B_\mathcal{S}\subset\mathcal{S}$ tel que: l'ensemble $\mathcal{S}\setminus B_\mathcal{S}$ est fini et pour tout $p\in B_\mathcal{S}$, la série $f(z)$ vérifie la propriété $(\textbf{P})_{p,r,n}$. 
Dans ce qui suit nous fixons l'ensemble $\mathcal{S}$, la série $f(z)\in\mathcal{F}(\mathcal{S})$ et les entiers $n$ et $r$. Soient $\gamma_1,\ldots,\gamma_r$ les singularités à distance finie de $\mathcal{M}_f$ dans $\overline{\mathbb{Q}}$ et écrivons $\mathcal{M}_f=\frac{d^n}{dz^n}+a_1(z)\frac{d^{n-1}}{dz^{n-1}}+\cdots+a_{n-1}(z)\frac{d}{dz}+a_n(z)\in\mathbb{Q}(z)[d/dz].$ Soit $d=\prod\limits_{i\neq j}(\gamma_i-\gamma_j)$. Alors $d$ est un nombre algébrique différent de zéro. Nous posons $B_{\mathcal{S}}$ comme l'ensemble des nombres premiers $p$ contenu dans $\mathcal{S}$ tel que la norme $p$-adique de $\gamma_i$ pour tout $\gamma_i\neq0$ est égale à 1, la norme de Gauss des $a_i(z)$ est inférieure ou égale à 1 et la norme $p$-adique de $d$ est égale à 1.

\begin{remarque}\label{rem_prop_b_s}
L'ensemble $\mathcal{S}\setminus B_{\mathcal{S}}$ est fini et par construction, pour tout $p\in B_{\mathcal{S}}$, $\mathcal{M}_f$ est à coefficients dans $E_{0,p}\cap\vartheta_{E_p}\cap\mathbb{Q}(z)$.  Donc, la remarque~\ref{rem_g_n} implique que, pour tout $p\in B_{\mathcal{S}}$, $\mathcal{M}_{f,\delta}$ est à coefficients dans $E_{0,p}\cap\vartheta_{E_p}\cap\mathbb{Q}(z)$.

 \end{remarque}



\begin{lem}\label{Premier pas}
Soit $\mathcal{S}$ un ensemble infini de nombres premiers et $f(z)\in\mathcal{F}(\mathcal{S})$. Si $f(z)$ annule $\mathcal{D}\in\mathbb{Q}(z)[d/dz]$ MOM en zéro alors, pour tout $p\in B_{\mathcal{S}}$, la série $f(z)$ vérifie la propriété $(\textbf{P})_{p,r,n}$. En outre, pour tout $p\in B_{\mathcal{S}}$, l'opérateur différentiel $\mathcal{M}_f$ satisfait aux conditions (1)-(6) de la définition~\ref{propriedadP}.
\end{lem}

\begin{proof} On va montrer que pour tout $p\in B_{\mathcal{S}}$, l'opérateur $\mathcal{M}_f$ satisfait aux conditions (1)-(6) de la définition~\ref{propriedadP}. Soit $p\in B_{\mathcal{S}}$. Il est clair que $\mathcal{M}_f$ est annulé par $f$ et que l'ordre de $\mathcal{M}_f$ est $n$. Soit $\mathcal{M}_{f,\delta}$ l'opérateur différentiel obtenu après avoir réécrit $z^n\mathcal{M}_f$ en fonction de $\delta$ et soit $M_{f,\delta}$ la matrice compagnon de $\mathcal{M}_{f,\delta}$. Il suit du lemme~\ref{lem_minimal} que, pour tout $p\in B_\mathcal{S}$,  $M_{f,\delta}$ appartient à $\mathscr{M}_p$. Maintenant nous montrons que pour $p\in B_\mathcal{S}$, $\mathcal{M}_f$ est $p$-unipotent. En effet, comme $M_{f,\delta}$ appartient à $\mathscr{M}_p$ alors le système $\delta X=M_{f,\delta}X$ a une base de solutions dans le disque générique de rayon 1, où $M_f$ est la matrice compagnon de $\mathcal{M}_f$. Donc, il suit de la remarque~\ref{rem_g_n} que le système $d/dz X=M_{f}X$ a aussi une base de solutions dans le disque générique de rayon 1. Ainsi, d'après la proposition~5.1 de \cite[Chap~III]{Dworkgfunciones}, $\mathcal{M}_{f,p}$ est nilpotent \footnote{Un opérateur différentiel $\mathcal{D}_p$ à coefficients dans $k(z)$, $k$ un cors fini de caractéristique $p$, est nilpotent si la $p$-courbure de $\mathcal{D}_p$ est nilpotente.}. Donc, d'après le théorème de Katz--Honda (voir théorème 2.3 de \cite[Chap~III]{Dworkgfunciones}) l'opérateur $\mathcal{M}_{f,p}$ est fuchsien. Ainsi, $\mathcal{M}_f$ est $p$-unipotent. Montrons que $M_f$ a comme uniques disques singuliers à distance finie les disques $D_{\gamma_1},\ldots D_{\gamma_r}$. Soit $D_{\gamma}$ un disque singulier à distance finie de $M_f$. Donc, le disque $D_{\gamma}$ contient une singularité de $\mathcal{M}_f$. Disons que $\gamma_i\in D_{\gamma}$ pour un certain $i\in\{1,\ldots,r\}$. Ainsi, $D_{\gamma}\cap D_{\gamma_i}\neq\emptyset$. Par conséquent, $D_{\gamma}=D_{\gamma_i}$. Maintenant, il suffit de montrer que $D_{\gamma_i}\cap D_{\gamma_j}=\emptyset$ si $i\neq j$. En effet, comme $p\in B _{\mathcal{S}}$ alors pour tout $\gamma_i\neq0$,  $|\gamma_i|_{p}=1$ et ainsi, pour tout $i,j\in\{1,\ldots,r\}$, $|\gamma_i-\gamma_j|_p\leq1$. Maintenant, comme $|d|_p=1$ alors $\prod\limits_{i\neq j}|\gamma_i-\gamma_j|_p=1$, d'où  $|\gamma_i-\gamma_j|_p=1$ pour $i\neq j$. Ainsi, $D_{\gamma_i}\cap D_{\gamma_j}=\emptyset$. Finalement, l'opérateur $\mathcal{M}_f$ est fuchsien car pour tout $p\in B_{\mathcal{S}}$, $\mathcal{M}_{f,p}$ est nilpotent donc, d'après le théorème 6.1 de \cite[Chap~III]{Dworkgfunciones}, l'opérateur $\mathcal{M}_{f}$ est fuchsien. Ainsi, la proposition~6.1.3 de \cite{Gillesmoduldiff} entraîne que les disques singuliers de $M_f$ sont singuliers réguliers. Raison pour laquelle, pour tout $p\in B_{\mathcal{S}}$ la série $f(z)$ vérifie la propriété $(\textbf{P})_{p,r,n}$ car l'opérateur $\mathcal{M}_f$ satisfait aux conditions (1)-(6) de la définition~\ref{propriedadP}.  
\end{proof}

En admettant le lemme~\ref{Deuxième pas}  nous sommes en mesure de démontrer la proposition~\ref{recurrence}.  

\subsubsection{Démonstration de la proposition \ref{recurrence}}\label{dem}

\begin{proof} 
Soit $\mathcal{S'}=B_{\mathcal{S}}$. Alors $\mathcal{S}'\subset\mathcal{S}$ est infini et $\mathcal{S}\setminus \mathcal{S}'$ est fini. D'après le lemme~\ref{Premier pas}, pour tout $p\in\mathcal{S}'$ la série $f(z)$ vérifie la propriété $(\textbf{P})_{p,r,n}$.	Ainsi, d'après le lemme~\ref{Deuxième pas}, pour tout $p\in\mathcal{S}'$ et tout entier $k$ positif la série $\Lambda^k_p(f(z))$ vérifie la propriété $(\textbf{P})_{p,r,n}$. En particulier $\Lambda^k_p(f(z))$ annule un opérateur $p$-unipotent $\mathcal{L}_k\in\vartheta_{E_p}[d/dz]$ qui a exactement $r$ disques singuliers réguliers à distance finie. Donc, d'après le lemme~\ref{singularites}, le nombre de singularités à distance finie de $\mathcal{L}_{k,p}$ dans $\overline{\mathbb{F}_p}$ est inférieur ou égal à $r$.
\end{proof}
\begin{remarque}\label{S'}
	Soient $\mathcal{S}$ un ensemble infini de nombres premiers et $f(z)\in\mathcal{F}(\mathcal{S})$. Supposons que $f(z)$ annule un opérateur différentiel $\mathcal{D}\in\mathbb{Q}(z)[d/dz]$ MOM en zéro. D'après la première partie du théorème~\ref{practico}, il existe $\mathcal{S}'\subset\mathcal{S}$ infini tel que $f(z)\in\mathcal{L}^2(\mathcal{S}')$ et $\mathcal{S}\setminus\mathcal{S}'$ est fini. Il suit des parties~\ref{construction} et \ref{dem} que nous pouvons prendre $\mathcal{S}'= B_{\mathcal{S}}$.
\end{remarque}

\subsubsection{Démonstration du lemme~\ref{Deuxième pas}}\label{sec_dem_deuxieme_pas}

\smallskip

Pour la démonstration du lemme~\ref{Deuxième pas} nous avons besoin d'un résultat préliminaire. Avant d'énoncer ce résultat nous avons la remarque suivante.
\begin{remarque}\label{sol}
	Soit $G(z)\in\mathscr{M}_p$ de taille $n$. Alors le système $\delta X=G(z)X$ à une base de solutions dans l'anneau $\mathbb{C}_p[[z,Logz]]$, où $\delta Logz=1$. En effet, comme les valeurs propres de la matrice $G(0)$ sont toutes égales à zéro et $G(z)$ est à coefficients dans $\mathbb{C}_p[[z]]$ alors, d'après la proposition~8.5 de \cite[Chap~III]{Dworkgfunciones}, une matrice fondamentale de solutions du système $\delta X=G(z)X$ est donné par $Y_GX^{G(0)}$, où $Y_G\in Gl_n(\mathbb{C}_p[[z]])$, $Y_G(0)$ est la matrice identité et $X^{G(0)}=\sum_{j\geq0}G(0)^j\frac{(Logz)^j}{j!}$. Comme $\delta Logz=1$ alors $\delta X^{G(0)}=G(0)X^{G(0)}$. Ainsi, $\delta(Y_G)+G(z)G(0)=G(z)Y_G(z)$ car $\delta(Y_GX^{G(0)})=G(z)Y_GX^{G(0)}$. Comme toutes les valeurs propres de $G(0)$ sont égales à zéro alors $G(0)^n=0$ et ainsi, $X^{G(0)}=\sum_{j=0}^{n-1}G(0)^j\frac{(Logz)^j}{j!}$. La matrice $Y_G$ est appelée \emph{la part uniforme} de la matrice $G$. 
\end{remarque}
 
 Nous rappelons qu'une matrice $G$ à coefficients dans $E_p$ de taille $n$ a une structure de Frobenius faible s'il existe un entier $h>0$ et deux matrices $H$ et $F$ de taille $n$ à coefficient dans $E_p$ telles que $H$ est inversible et $$\delta H(z)=G(z)H(z)-p^hH(z)F(z^{p^h}).$$ Le lemme suivant nous assure l'existence de la structure de Frobenius faible pour toute matrice dans $\mathscr{M}_p$.

\begin{lem}\label{10}
	 Si $G(z)$ appartient à $\mathscr{M}_p$ alors il existe $F(z)$ dans $\mathscr{M}_p$ telle que les matrices $G$ et $pF(z^p)$ sont équivalentes. Plus précisément on a:\begin{enumerate}
	 	\item  $F(z)=[\delta(\Lambda_p(Y_G))+\frac{1}{p}\Lambda_p(Y_G)G(0)](\Lambda_p(Y_G))^{-1}$, où $Y_G$ est la part uniforme de la matrice $G$ et $\Lambda_p(Y_G)$ est la matrice obtenue après avoir appliqué $\Lambda_p$ à chaque entrée de $Y_G$.
	 	\item La matrice $H=\Lambda_p(Y_G)(z^p)Y_G^{-1}$ appartient à $Gl_n(E_{0,p})$ et on a aussi $\delta H=pF(z^p)H-HG$.
	 \end{enumerate}
\end{lem}
 La démonstration que nous présentons de ce lemme suit les arguments donnés par Christol dans \cite[lemme~5.1]{christolpadique}. 
\begin{proof}
	Soit $G(z)$ dans $\mathscr{M}_p$ de taille $n$. D'après la remarque~\ref{sol}, une matrice fondamentale de solutions du système $\delta X=G(z)X$ est donnée par $Y_{G}X^{G(0)}$, où $Y_G\in Gl_n(\mathbb{C}_p[[z]])$ et $Y_G(0)=I_n$, où $I_n$ est la matrice identité de taille $n$. Maintenant considérons la suite de matrices $G_j$ définie comme suit: $G_0(z)=Id$ et $G_{j+1}(z)=\delta G_j(z)+G_j(z)(G(z)-jI_n)$. On pose $H(z)=\frac{1}{p}\sum_{\xi^p=1}\sum_{j\geq0}G_j(z)\frac{(\xi-1)^j}{j!}$. Comme $G(z)$ vérifie la condition iii) de la définition~\ref{R} alors, il est montré dans \cite[p.~164]{christolpadique} que la matrice $H$ appartient à $Gl_n(E_{0,p})$. Écrivons $Y_G=\sum_{j\geq0}Y_jz^j$, où pour tout $j\geq0$, $Y_j$ est une matrice à coefficients dans $\mathbb{C}_p$. Comme la matrice $G(0)$ est nilpotente car $G(0)^n=0$ (voir remarque~\ref{sol}) alors, il est montré dans \cite[p.~165]{christolpadique} que, $HY_G=\sum_{j\geq0}Y_{jp}z^{jp}$. Ainsi, $HY_G=\Lambda_p(Y_G)(z^p)$, où $\Lambda_p(Y_G)$ désigne la matrice obtenue après avoir appliqué $\Lambda_p$ à chaque entrée de $Y_G$. Par conséquent, $H(0)=I_n$. Maintenant on pose
	\begin{equation}\label{construcionF}
	F(z)=[\delta(\Lambda_p(Y_G))+\frac{1}{p}\Lambda_p(Y_G)G(0)](\Lambda_p(Y_G))^{-1}.
	\end{equation} 
	 Ainsi, $pF(z^p)=[p(\delta(\Lambda_p(Y_G)))(z^p)+HY_GG(0)][Y^{-1}_GH^{-1}]$. Notons que $$(\delta H)Y_G+H(\delta Y_G)=\delta(HY_G)=\delta(\Lambda_p(Y_G)(z^p))=p(\delta(\Lambda_p(Y_G)))(z^p).$$ D'après la remarque~\ref{sol}, $\delta Y_G=GY_G-Y_GG(0)$ alors, \[p(\delta(\Lambda_p(Y_G)))(z^p)=(\delta H)Y_G+H[GY_G-Y_GG(0)].\] Donc, 
	\begin{align*}
	pF(z^p)=&[(\delta H)Y_G+H[GY_G-Y_GG(0)]+HY_GG(0)][Y^{-1}_GH^{-1}]\\
	=&[(\delta H)Y_G+HGY_G][Y^{-1}_GH^{-1}]\\
	=&(\delta H)H^{-1}+HGH^{-1}.
	\end{align*}
	Par conséquent, $\delta H=pF(z^p)H-HG$. Il est montré dans \cite[p.165-166]{christolpadique} que la matrice $F(z)$ vérifie les conditions i) et iii) de la définition~\ref{R}. En particulier on peut évaluer la matrice $F$ en zéro. Comme $H(0)=I_n$ alors l'égalité $\delta H=pF(z^p)H-HG$ entraîne que, $pF(0)=G(0)$. Ainsi, toutes les valeurs propres de $F(0)$ son égales à zéro. Donc, la matrice $F$ vérifie la condition ii) de la définition~\ref{R}. Ainsi, la matrice $F(z)$ appartient à $\mathscr{M}_p$ et les matrices $pF(z^p)$ et $G(z)$ sont équivalentes. Finalement, comme $HY_G=\Lambda_p(Y_G)(z^p)$ alors $H=\Lambda_p(Y_G)(z^p)Y_G^{-1}.$
\end{proof}

\begin{proof}[Démonstration du lemme~\ref{Deuxième pas}]Nous allons faire la démonstration en cinq pas.

	\textbf{Premier pas.} \emph{Il existe un opérateur différentiel $\mathcal{L}_{1,\delta}\in E_{0,p}[\delta]$ d'ordre $n$ qui est annulé par $\Lambda_p(f(z))$}.	
		
	  Écrivons $\mathcal{L}_{\delta}:=\delta^n+ e_1(z)\delta^{n-1}+\cdots+e_{n-1}(z)\delta+e_{n}(z)$, où les $e_i(z)\in\mathbb{C}_p[[z]]$ pour $i\in\{1,\ldots,n\}$ car les $e_i(z)\in E_{0,p}$. Par définition,  \[
	 L_{\delta}(z)=\begin{pmatrix}
	 0 & 1 & 0 & \dots & 0 & 0\\
	 0 & 0 & 1 & \dots & 0 & 0\\
	 \vdots & \vdots & \vdots & \vdots & \vdots & \vdots\\
	 0 & 0 & 0 & \ldots & 0 & 1\\
	 -e_n(z) & -e_{n-1}(z) & -e_{n-3}(z) & \ldots & -e_2(z) & -e_1(z)\\
	 \end{pmatrix}.
	 \] 
	 Les valeurs propres de $L_{\delta}(0)$ sont les racines du polynôme  $X^n+e_1(0)X^{n-1}+\cdots+e_{n-1}(0)X+e_n(0)$ mais, par hypothèse, les valeurs propres de $L_{\delta}(0)$ sont toutes égales à zéro, donc $e_i(0)=0$. D'où \[
	L_{\delta}(0)=\begin{pmatrix}
	0 & 1 & 0 & \dots & 0 & 0\\
	0 & 0 & 1 & \dots & 0 & 0\\
	\vdots & \vdots & \vdots & \vdots & \vdots & \vdots\\
	0 & 0 & 0 & \ldots & 0 & 1\\
	0 & 0 & 0 & \ldots & 0 & 0\\
	\end{pmatrix}.
	\] 
	Soit $\boldsymbol{\delta}f=(f(z),\delta f(z),\ldots,\delta^{(n-1)}f(z))^t$. Comme par hypothèse $\mathcal{L}_{\delta}(f)=0$ alors $\delta(\boldsymbol{\delta}f))=L_{\delta}(z)\boldsymbol{\delta}f.$ D'après la remarque~\ref{sol}, une matrice fondamentale de solutions du système $\delta X=L_{\delta}(z)X$ est donnée par 
	\[
	Y_LX^{L_{\delta}(0)}=Y_{L}\cdot\begin{pmatrix}
	1 & Logz & \frac{(Logz)^2}{2!} & \dots & \frac{(Logz)^{n-1}}{(n-1)!} \\
	0 & 1 & Logz & \dots & \frac{(Logz)^{n-2}}{(n-2)!} \\
	\vdots & \vdots & \vdots & \vdots & \vdots \\
	0 & 0 & 0 & \ldots &  Logz\\
	0 & 0 & 0 & \ldots &  1\\
	\end{pmatrix},
	\]
	où $Y_L\in Gl_n(\mathbb{C}_p[[z]])$ est telle que $Y_L(0)$ est la matrice identité. Alors, le système différentiel $\delta X=L_{\delta}(z)X$ a une unique solution dans $\mathbb{C}_p[[z]]^n$ à constante près\footnote{Ici nous avons utilisé le fait que $Logz$ est transcendant sur $\mathbb{C}_p((z)).$}. Puisque $\boldsymbol{\delta}f\in\mathbb{C}_p[[z]]^n$ est solution du système $\delta X=L_{\delta}(z)X$ et $f(0)=1$ on obtient que la première colonne de la matrice $Y_{L}$ est le vecteur $\boldsymbol{\delta}f$. Comme $L_{\delta}(z)$ est dans $\mathscr{M}_p$, d'après le lemme~\ref{10}, il existe $F(z)$ dans $\mathscr{M}_p$ telle que les matrices $L_{\delta}(z)$ et $pF(z^p)$ sont équivalentes. De plus, il suit encore du lemme~\ref{10} que
	\begin{equation}\label{F}
	F(z)=[\delta(\Lambda_p(Y_L))+\frac{1}{p}\Lambda_p(Y_L)L_{\delta}(0)](\Lambda_p(Y_L))^{-1},
	\end{equation}
	où $\Lambda_p(Y_L)$ désigne la matrice obtenue après avoir appliqué $\Lambda_p$ à chaque entrée de $Y_L$. Soit $I_n$ la matrice identité de taille $n$. Notons que $\Lambda_p(Y_L)(0)=I_n$ car $Y_L(0)=I_n$ donc, d'après \eqref{F}, on a 
	\begin{equation}\label{F0}
	F(0)=\frac{1}{p}L_{\delta}(0)=\begin{pmatrix}
	0 & \frac{1}{p} & 0 & \dots & 0 & 0\\
	0 & 0 & \frac{1}{p} & \dots & 0 & 0\\
	\vdots & \vdots & \vdots & \vdots & \vdots & \vdots\\
	0 & 0 & 0 & \ldots & 0 & \frac{1}{p}\\
	0 & 0 & 0 & \ldots & 0 & 0\\
	\end{pmatrix}.
	\end{equation}
	Considérons le produit
	\begin{equation}\label{matricefondamentalF}
	\Lambda_p(Y_L)X^{F(0)}=\Lambda_p(Y_{L})\cdot\begin{pmatrix}
	1 & \frac{Logz}{p} & \frac{1}{p^2}\frac{(Logz)^2}{2!} & \dots & \frac{1}{p^{n-1}}\frac{(Logz)^{n-1}}{(n-1)!} \\
	0 & 1 & \frac{1}{p}Logz & \dots & \frac{1}{p^{n-2}}\frac{(Logz)^{n-2}}{(n-2)!} \\
	\vdots & \vdots & \vdots & \vdots & \vdots \\
	0 & 0 & 0 & \ldots &  \frac{1}{p}Logz\\
	0 & 0 & 0 & \ldots &  1\\
	\end{pmatrix}.
	\end{equation}
	 Comme $\delta X^{F(0)}=F(0)X^{F(0)}$ alors, il suit de  \eqref{F} et \eqref{F0} que $$\delta(\Lambda_p(Y_L)X^{F(0)})=F(z)\Lambda_p(Y_L)X^{F(0)}.$$ Par conséquent, la matrice $\Lambda_p(Y_L)X^{F(0)}$ est une matrice fondamentale de solutions du système différentiel $\delta X=F(z)X$. Comme la première colonne de $Y_L$ est donnée par le vecteur $$\boldsymbol{\delta}f=(f(z),\delta f(z),\ldots,\delta^{(n-1)}f(z))^t$$ alors la première colonne de $\Lambda_p(Y_L)$ est le vecteur $$\Lambda_p(\boldsymbol{\delta}f)=(\Lambda_p(f(z)),\Lambda_p(\delta f(z)),\cdots,\Lambda_p(\delta^{(n-1)}f(z)))^t.$$ Ainsi, $\Lambda_p(\boldsymbol{\delta}f)$ est solution du système différentiel $\delta\vec{y}=F(z)\vec{y}$. Écrivons $F(z)=(a_{i,j}(z))_{1\leq i,j\leq n}$, où les $a_{i,j}(z)$ appartiennent à $E_{0,p}$ car $F(z)$ est dans $\mathscr{M}_p$. Alors,
	 \begin{small}

	 \[\begin{pmatrix}
	 a_{1,1}(z) & a_{1,2}(z) & \cdots & a_{1,n}(z)\\
	 a_{2,1}(z) & a_{2,2}(z) & \cdots & a_{2,n}(z)\\
	 \vdots & \vdots & \vdots & \vdots  \\
	 a_{n,1}(z) & a_{n,2}(z) & \cdots & a_{n,n}(z)\\
	 \end{pmatrix}\begin{pmatrix}
	 \Lambda_p(f(z))\\
	 \Lambda_p(\delta(f(z)))\\
	 \vdots\\
	 \Lambda_p(\delta^{n-1}(f(z)))
	 \end{pmatrix}=\begin{pmatrix}
	 \delta(\Lambda_p(f(z)))\\
	 \delta(\Lambda_p(\delta(f(z))))\\
	 \vdots\\
	 \delta(\Lambda_p(\delta^{n-1}(f(z))))
	 \end{pmatrix}.\]
	\end{small}
	 Par conséquent,
	 	\begin{multline}
	 	a_{n,1}(z)\Lambda_p(f(z))+a_{n,2}(z)\Lambda_p(\delta f(z))+\cdots+\\+a_{n,n-1}(z)\Lambda_p(\delta^{n-2}f(z))+a_{n,n}(z)\Lambda_p(\delta^{n-1}f(z))=\delta(\Lambda_p(\delta^{n-1}f(z))).
	 	\end{multline}
	  Mais $\Lambda_p(\delta^jf(z))=p^j\delta^j\Lambda_p(f(z))$, alors $\Lambda_p(f(z))$ est solution de l'opérateur différentiel $$\mathcal{L}_{1,\delta}=\delta^n-a_{n,n}(z)\delta^{n-1}(z)-\frac{a_{n,n-1}(z)}{p}\delta^{n-2}-\cdots-\frac{a_{n,2}(z)}{p^{n-2}}\delta-\frac{a_{n,1}(z)}{p^{n-1}}.$$
	 Notons que $\mathcal{L}_{1,\delta}$ est à coefficients dans $E_{0,p}$ car $a_{n,n}(z),\ldots, a_{n,1}(z)\in E_{0,p}$.
	 
	 \smallskip
	
	\textbf{Deuxième pas.} \emph{ Soit $L_{1,\delta}(z)$ la matrice compagnon de $\mathcal{L}_{1,\delta}$. Alors $L_{1,\delta}(z)$ est dans $\mathscr{M}_p$}. Par définition la matrice compagnon de $\mathcal{L}_{1,\delta}$ est \[
	L_{1,\delta}(z)=\begin{pmatrix}
	0 & 1 & 0 & \dots & 0 & 0\\
	0 & 0 & 1 & \dots & 0 & 0\\
	\vdots & \vdots & \vdots & \vdots & \vdots & \vdots\\
	0 & 0 & 0 & \ldots & 0 & 1\\
	\frac{a_{n,1}(z)}{p^{n-1}} & \frac{a_{n,2}(z)}{p^{n-2}} & \frac{a_{n,3}(z)}{p^{n-3}} & \ldots & \frac{a_{n,n-1}(z)}{p} & a_{n,n}(z)\\
	\end{pmatrix}
	\] Comme $a_{n,1}(z),\ldots, a_{n,n}(z)$ sont dans $E_{0,p}$ alors $L_{1,\delta}(z)$ est une matrice à coefficients dans $E_{0,p}$.  D'après \eqref{F0}, $a_{n,1}(0)=\cdots=a_{n,n}(0)=0$ donc, toutes les valeurs propres de $L_{1,\delta}(0)$ sont égales à zéro et ainsi la matrice $L_{1,\delta}(z)$ satisfait la condition ii) de la définition~\ref{R}. Il nous reste à voir que $L_{1,\delta}(z)$ vérifie la condition iii) de la définition~\ref{R}. À la suite de la remarque~\ref{sol} on a qu'une matrice fondamentale de solutions du système $\delta X=L_{1,\delta}(z)X$ est donnée par \[
	Y_{L_1}X^{L_{1,\delta}(0)}=Y_{L_1}\cdot\begin{pmatrix}
	1 & Logz &\frac{(Logz)^2}{2!} & \dots & \frac{(Logz)^{n-1}}{(n-1)!} \\
	0 & 1 & Logz & \dots &\frac{(Logz)^{n-2}}{(n-2)!} \\
	\vdots & \vdots & \vdots & \vdots & \vdots \\
	0 & 0 & 0 & \ldots & Logz\\
	0 & 0 & 0 & \ldots &  1\\
	\end{pmatrix},
	\] où $Y_{L_1}\in Gl_n(\mathbb{C}_p[[z]])$ est telle que $Y_{L_1}(0)$ est la matrice identité. Écrivons $Y_{L_1}=(g_{i,j}(z))_{1\leq i,j\leq n}$. Comme la matrice $L_{1,\delta}(z)$ est à coefficients dans $E_{0,p}$ et toutes le valeurs propres de $L_{1,\delta}(0)$ sont égales à zéro alors, d'après le théorème~2 de \cite{christolpadique}, pour montrer que $L_{1,\delta}(z)$ vérifie la condition iii) de la définition~\ref{R}, il suffit de voir que le rayon de convergence des séries $g_{i,j}(z)$ est supérieur ou égal à 1. Écrivons $Y_{L}=(f_{i,j}(z))_{1\leq i,j\leq n}$. Par hypothèse $L_{\delta}(z)$ est dans $\mathscr{M}_p$ alors, d'après le théorème~2 de \cite{christolpadique}, le rayon de convergence des séries $f_{i,j}(z)$ est supérieur ou égal à 1 pour $i,j\in\{1,\ldots,n\}$. Comme la norme est ultramétrique alors le rayon de convergence des séries $\Lambda_p(f_{i,j}(z))$ est supérieur ou égal à 1. Alors, pour montrer que le rayon de convergence des séries $g_{i,j}(z)$ est supérieur ou égal à 1 il suffit de montrer que $$Y_{L_1}=diag(1,1/p,\ldots, 1/p^{n-1})\Lambda_p(Y_{L})diag(1,p,\ldots, p^{n-1}).$$
Comme $diag(1,1/p,\ldots, 1/p^{n-1})\Lambda_p(Y_{L})diag(1,p,\ldots, p^{n-1})(0)$ est la matrice l'identité, cela revient à montrer que la matrice $$T=diag(1,1/p,\ldots, 1/p^{n-1})\Lambda_p(Y_{L})diag(1,p,\ldots, p^{n-1})X^{L_{1,\delta}(0)}$$  est une matrice fondamentale de solutions du système $\delta X=L_{1,\delta}(z)X$.  D'abord, on montre que, pour tout $i\in\{1,\ldots,n-1\}$ et tout $k\in\{1,\ldots,n\}$, $f_{i+1,k}=f_{i,k-1}+\delta f_{j,k}$.
	En effet, pour chaque $i\in\{1,\ldots,n\}$, considérons $$F_{i,j}=\sum_{k=1}^jf_{i,k}\frac{(Logz)^{j-k}}{(j-k)!}.$$
	Remarquons que $Y_{L}X^{L_{\delta}(0)}=(F_{i,j})_{1\leq i,j\leq n}$. Comme $Y_{L}X^{L_{\delta}(0)}$ est une matrice fondamentale de solutions de $\delta X=L_{\delta}(z)X$ et $L_{\delta}(z)$ est la matrice compagnon de $\mathcal{L}_{\delta}$ alors, pour tout $i\in\{2,\ldots, n\}$ et  $j\in\{1,\ldots, n\}$, $F_{i,j}=\delta F_{i-1,j}$. Mais, il est clair que $$\delta F_{i,j}
=\delta(f_{i,1})\frac{(Logz)^{j-k}}{(j-k)!}+\sum_{k=2}^j(f_{i,k-1}+\delta(f_{i,k}))\frac{(Logz)^{j-k}}{(j-k)!}.$$  Puisque $F_{i+1,j}=\delta F_{i,j}$ et $Logz$ est transcendant sur $\mathbb{C}_p[[z]]$, l'égalités précédentes impliquent que, pour tout $k\in\{1,\ldots, n\}$, $f_{i+1,k}=f_{i,k-1}+\delta f_{j,k}$.
	
Maintenant, nous allons montrer que $T$ est une matrice fondamentale de solutions de $\delta X=L_{1,\delta}(z)X$. Pour cela, nous montrons d'abord que, pour chaque $j\in\{1,\ldots, n\}$, $$\omega_j=\sum_{k=1}^{j}p^{k-1}\Lambda_p(f_{1,k})\frac{(Logz)^{j-k}}{(j-k)!}$$
est une solution de $\mathcal{L}_{1,\delta}$. Notons que le vecteur $(\omega_1,\ldots,\omega_n)$ est la première ligne de la matrice $T$. On pose $\Lambda_p(Y_L)X^{\frac{1}{p}L_{\delta}(0)}=(\eta_{i,j})_{1\leq i,j\leq n}.$ Donc, pour tout $i,j\in\{1,\ldots,n\}$,  $$\eta_{i,j}=\sum_{k=1}^j\Lambda_p(f_{i,k})\frac{(Logz)^{j-k}}{p^{j-k}(j-k)!}.$$
On va voir que, pour tout $j,l\in\{1,\ldots,n\}$, $\frac{1}{p^{n-l}}\delta^{l-1}\omega_j=\frac{1}{p^{n-j}}\eta_{l,j}$. Pour montrer cette égalité, on procède par induction sur $l\in\{1,\ldots, n\}$. Pour $l=1$, il est clair que $\frac{1}{p^{n-1}}\delta^{l-1}\omega_j=\frac{1}{p^{n-j}}\eta_{1,j}$. Maintenant, on suppose que pour certain $l\in\{1,\ldots, n\}$, $\frac{1}{p^{n-l}}\delta^{l-1}\omega_j=\frac{1}{p^{n-j}}\eta_{l,j}$. Donc, $\frac{1}{p^{n-l}}\delta^{l}\omega_j=\frac{1}{p^{n-j}}\delta(\eta_{l,j})$. Mais, $\delta(\eta_{l,j})=\frac{1}{p}\eta_{l+1,j}$. En effet, comme  $\Lambda_p\circ\delta=p\delta\circ\Lambda_p$ et $f_{l,k-1}+\delta(f_{l,k})=f_{l+1,k}$ pour tout $k\in\{1,\ldots, n\}$ alors
\begin{align*}
\delta(\eta_{l,j})&=\sum_{k=1}^j\delta(\Lambda_p(f_{l,k}))\frac{(Logz)^{j-k}}{p^{j-k}(j-k)!}+\Lambda_p(f_{l,k})\frac{(Logz)^{j-k-1}}{p^{j-k}(j-k-1)!}\\
&=\sum_{k=1}^j\frac{1}{p}\Lambda_p(\delta(f_{l,k}))\frac{(Logz)^{j-k}}{p^{j-k}(j-k)!}+\frac{1}{p}\Lambda_p(f_{l,k})\frac{(Logz)^{j-k-1}}{p^{j-k-1}(j-k-1)!}\\
&=\frac{1}{p}\left[\Lambda_p(\delta(f_{l,1}))\frac{(Logz)^{j-1}}{p^{j-1}(j-1)!}+\sum_{k=2}^j\Lambda_p(f_{l,k-1}+\delta f_{l,k}))\frac{(Logz)^{j-k}}{p^{j-k}(j-k)!}\right]\\
&=\frac{1}{p}\left[\Lambda_p(f_{l+1,1})\frac{(Logz)^{j-1}}{p^{j-1}(j-1)!}+\sum_{k=2}^j\Lambda_p(f_{l+1,k})\frac{(Logz)^{j-k}}{p^{j-k}(j-k)!}\right]\\
&=\frac{1}{p}\eta_{l+1,j}.
\end{align*}
Ainsi,  $\frac{1}{p^{n-l}}\delta^{l}\omega_j=\frac{1}{p^{n-j}}\delta(\eta_{l,j})=\frac{1}{p^{n-j}}(\frac{1}{p}\eta_{l+1,j})$. D'où, $\frac{1}{p^{n-l-1}}\delta^{l}\omega_j=\frac{1}{p^{n-j}}\eta_{l+1,j}$. Pour cette raison, on conclude que,  pour tout $j,l\in\{1,\ldots,n\}$, $\frac{1}{p^{n-l}}\delta^{l-1}\omega_j=\frac{1}{p^{n-j}}\eta_{l,j}$.

Comme on a vu dans le premier pas, $\Lambda_p(Y_L)X^{\frac{1}{p}L_{\delta}(0)}$  es une matrice fondamentale de solutions de $\delta X=FX$. Donc, pour tout $j\in\{1,\ldots,n\}$, \[F\begin{pmatrix} 
\eta_{1,j}\\
\eta_{2,j}\\
\vdots\\
\eta_{n,j}\\
\end{pmatrix}
=\begin{pmatrix} 
\delta(\eta_{1,j})\\
\delta(\eta_{2,j})\\
\vdots\\
\delta(\eta_{n,j}))\\
\end{pmatrix}.
\]
Par conséquent, pour tout $j\in\{1,\ldots,n\}$, $$b_{n,1}\eta_{1,j}+b_{n,2}\eta_{2,j}+\cdots+b_{n,k}\eta_{k,j}+\cdots+b_{n,n}\eta_{n,j}=\delta(\eta_{n,j}).$$
Ainsi,  $$\frac{1}{p^{n-j}}\left[b_{n,1}\eta_{1,j}+b_{n,2}\eta_{2,j}+\cdots+b_{n,k}\eta_{k,j}+\cdots+b_{n,n}\eta_{n,j}\right]=\frac{1}{p^{n-j}}\delta(\eta_{n,j}).$$
Mais, on sait que, pour tout $l\in\{1,\ldots, n\}$, $\frac{1}{p^{n-l}}\delta^{l-1}\omega_j=\frac{1}{p^{n-j}}\eta_{l,j}$. Donc, $$\frac{1}{p^{n-1}}b_{n,1}\omega_{j}+\frac{1}{p^{n-2}}b_{n,2}\delta\omega_{j}+\cdots+\frac{1}{p^{n-k}}b_{n,k}\delta^{k-1}\omega_{j}+\cdots+b_{n,n}\delta^{n-1}\omega_{j}=\delta^n\omega_j.$$
Cela entraîne que $\omega_j$ est une solution de $\mathcal{L}_{1,\delta}$. De plus, $\omega_1,\ldots,\omega_n$ sont linéairement indépendante sur $\mathbb{C}_p$ car  $Logz$ est transcendant sur $\mathbb{C}_p[[z]]$. Puisque $L_{1,\delta}$ est la matrice compagnon de $\mathcal{L}_{1,\delta}$, il s'ensuit que,  la matrice $(\delta^{i-1}\omega_j)_{1\leq i,j\leq n}$ est une matrice fondamentale de solutions du système $\delta X=L_{1,\delta}X$. Mais, il est facile à voir que $T=(\delta^{i-1}\omega_j)_{1\leq i,j\leq n}$. Par conséquent, $T$ est une matrice fondamentale de solutions du système $\delta X=L_{1,\delta}(z)X$.

\smallskip

\textbf{Troisième pas.} \emph{L'opérateur différentiel $\mathcal{L}_{1}\in E_p[d/dz]$ est $p$-unipotent}. Nous rappelons que $\mathcal{L}_{1}$ est l'opérateur différentiel obtenu après avoir réécrit $\frac{1}{z^n}\mathcal{L}_{1,\delta}$ en termes de $d/dz$. D'après le deuxième pas,  $L_{1,\delta}(z)$ appartient à $\mathscr{M}_p$. En particulier, il existe une matrice inversible $\widehat{U}_t(z)$ dont les coefficients sont analytiques dans le disque générique ouvert $D(t,1)$, telle que $\delta\widehat{U}=L_{1,\delta}(z)\widehat{U}$. D'après la remarque~\ref{rem_g_n}, on a $L_1(z)G_{n}=\frac{d}{dz}G_{n}+G_{n}\frac{1}{z}L_{1,\delta}$. Soit  $U_t=G_n\widehat{U}_t$. Donc, $U_t$ est inversible à coefficients analytiques dans le disque générique ouvert $D(t,1)$ et $\frac{d}{dz}U=L_1(z)U$. Comme $L_1$ est la matrice compagnon de $\mathcal{L}_1$ alors $\mathcal{L}_1$ a une base de solution dans le disque générique de rayon 1. Par conséquent, grâce au théorème de Frobenius--Dwork (voir \cite[Proposition~8.1]{gillesff}), on obtient que les normes des coefficients de l'opérateur différentiel $\mathcal{L}_{1}$ sont inférieures ou égales à 1. Autrement dit, $\mathcal{L}_{1}$ est à coefficients dans $\vartheta_{E_p}$. Notons $\mathcal{L}_{1,p}$ l'opérateur obtenu après avoir réduit chaque coefficient de $\mathcal{L}_1$ modulo l'idéal maximal $\mathfrak{m}_p$ de $\vartheta_{E_p}$. Comme le rayon de convergence de $\mathcal{L}_{1}$ au point générique est égal à 1 il suit alors de la proposition~5.1 de \cite[Chap~III]{Dworkgfunciones}  que $\mathcal{L}_{1,p}$ est nilpotent. Donc, d'après le théorème de Katz--Honda (voir théorème 2.3 de \cite[Chap~III]{Dworkgfunciones}) l'opérateur $\mathcal{M}_{f,p}$ est fuchsien. Ainsi, $\mathcal{M}_f$ est $p$-unipotent. Nous avons vu dans le deuxième pas que la matrice $L_{1,\delta}(z)$ appartient à $\mathscr{M}_p$ et par conséquent, tous les exposants en zéro de $\mathcal{L}_{1}$ sont tous égaux à zéro. Ainsi,  les exposants en zéro de $\mathcal{L}_{1,p}$ sont tous égaux à zéro. Ainsi, l'opérateur $\mathcal{L}_{1,p}$ est MOM en zéro. Il suit donc que $\mathcal{L}_1$ est $p$-unipotent.

\smallskip
		
\textbf{Quatrième pas.} \textit{Les matrices $L_{\delta}(z)$ et $pL_{1,\delta}(z^p)$ sont $E_p$-équivalentes}.   Considérons la matrice \[\widetilde{H}=Y_{L}(\Lambda_p(Y_L)(z^{p}))^{-1}diag(1,p,\ldots, p^{n-1}).\]
Par le lemme~\ref{10}, la matrice $H=Y_L(\Lambda_p(Y_L)(z^p))^{-1}$ appartient à $GL_n(E_{0,p})$. Ainsi, $\widetilde{H}_1$ appartient aussi à $GL_n(E_{0,p})$.  Nous allons voir que  $$\delta(\widetilde{H})=L_{\delta}(z)\widetilde{H}-p\widetilde{H}L_{1,\delta}(z^{p}).$$
On a montré dans le deuxième pas que la matrice $T(z)$ est une matrice fondamentale de solutions de $\delta X= L_{1,\delta}(z)X$. Donc, $T(z^p)$ est une matrice fondamentale de solutions de  $\delta X= pL_{1,\delta}(z^p)X$. Par définition de la matrice $T(z)$, on a $$T(z^p)=diag(1,1/p,\ldots, 1/p^{n-1})\Lambda_p(Y_{L})(z^p)diag(1,p,\ldots, p^{n-1})X^{pL_{1,\delta}(0)}.$$
 
 Ainsi, \[\widetilde{H}T(z^p)=Y_Ldiag(1,p,\ldots, p^{n-1})X^{pL_{1,\delta}(0)}.\] 
Mais, $diag(1,p,\ldots, p^{n-1})X^{pL_{1,\delta}(0)}=X^{L_{1,\delta}(0)}diag(1,p,\ldots, p^{n-1})$ et comme $L_{1,\delta}(0)=L_{\delta}(0)$ alors, on a \[\widetilde{H}T(z^p)=Y_LX^{L_{\delta}(0)}diag(1,p,\ldots, p^{n-1}). 
\]
Puisque $Y_LX^{L_{\delta}(0)}$ est une matrice fondamentale de solutions de $\delta X=L_{\delta}(z)X$, $\widetilde{H}T(z^p)$ est aussi une matrice fondamentale de solutions de $\delta X=L_{\delta}(z)X$. Ainsi, $L_{\delta}(z)\widetilde{H}T(z^p)=\delta(\widetilde{H}T(z^p))=\delta(\widetilde{H})T(z^p)+\widetilde{H}(pL_{1,\delta}(z^p))T(z^p)$.
D'où, $$\delta(\widetilde{H})=L_{\delta}(z)\widetilde{H}-p\widetilde{H}L_{1,\delta}(z^{p}).$$

\smallskip
\textbf{Cinquième pas.} \textit{Les disques singuliers de $L_1(z)$ sont singuliers réguliers et  $L_1(z)$ a exactement $r$ disques singuliers réguliers à distance finie}.
Nous rappelons que $L_1(z)$ est la matrice compagnon de $\mathcal{L}_1$ et que $\mathcal{L}_1$ est l'opérateur différentiel obtenu après avoir réécrit $\frac{1}{z^n}\mathcal{L}_{1,\delta}$ en termes de $d/dz$. Montrons d'abord que les matrices $\frac{p}{z}L_{1,\delta}(z^p)$ et $pz^{p-1}L_1(z^p)$ sont $E_p$-équivalentes. D'après la remarque~\ref{rem_g_n}, on a $L_1(z)G_{n}=\frac{d}{dz}G_{n}+G_{n}\frac{1}{z}L_{1,\delta}$. Par conséquent, $d/dz( G_{n}(z^p))=pz^{p-1}L_1(z^p)G_n(z^p)-G_n(z^p)\frac{p}{z}L_{1,\delta}(z^p)$. Ainsi, la matrice $\frac{p}{z}L_{1,\delta}(z^p)$ et $E_p$-équivalente à la matrice $pz^{p-1}L_1(z^p)$. Le quatrième pas entraîne que $\frac{1}{z}L_{\delta}(z)$ est $E_p$-équivalente à $\frac{p}{z}L_{1,\delta}(z^p)$. Alors, par transitivité la matrice $\frac{1}{z}L_{\delta}(z)$ est $E_p$-équivalente à $pz^{p-1}L_1(z^p)$. Rappelons que $L(z)$ est la matrice compagnon de $\mathcal{L}$. À nouveau par la remarque~\ref{rem_g_n}, on a $d/dz G_{n}=L(z)G_n(z)-G_n\frac{1}{z}L_{\delta}(z)$. Donc, les matrices  $\frac{1}{z}L_{\delta}(z)$ et $L(z)$ sont $E_p$-équivalentes. Ainsi par transitivité, les matrices  $pz^{p-1}L_1(z^p)$ et $L(z)$ sont $E_p$-équivalentes. Par hypothèse tous les disques singuliers de $L(z)$ sont singuliers réguliers et $L(z)$ a exactement $r$ disques singuliers réguliers à distance finie. Comme $L(z)$ est $E_p$ équivalente à $pz^{p-1}L_1(z^p)$ alors il suit de la remarque~\ref{trans} que tous les disque singuliers de la matrice $pz^{p-1}L_1(z^p)$ sont singuliers réguliers et que cette matrice a exactement $r$ disques singuliers à distance finie. 
  
  Comme on l'a déjà montré dans le troisième pas, il existe une matrice inversible $U$, dont les coefficients sont analytiques dans le disque générique ouvert $D(t,1)$, telle que $\frac{d}{dz}U=L_{1}(z)U$. Donc, il découle du corollaire~6.4.2 et de la propositions~6.4.6 de \cite{Gillesmoduldiff} que, $D_{\gamma}$ est un disque singulier régulier de $L_1(z)$ si et seulement si $D_{\gamma}$ est un disque singulier régulier de $pz^{p-1}L_1(z^p)$. Comme on l'a vu tous les disques singuliers de $pz^{p-1}L_1(z^p)$ sont singuliers réguliers et cette matrice a exactement $r$ disques singuliers à distance finie. Par conséquent, les disques singuliers de la matrice $L_1(z)$ sont singuliers réguliers et $L_1(z)$ a exactement $r$ disques singuliers à distance finie.\\
  
Pour finir la démonstration du lemme~\ref{Deuxième pas}, montrons que l'opérateur $\mathcal{L}_1$ obtenu après avoir réécrit $1/z^n\mathcal{L}_{1,\delta}$ en termes de $d/dz$ vérifie les conditions (1)-(6) de la définition~\ref{propriedadP}. Il suit du premier pas que l'opérateur $\mathcal{L}_1$ est d'ordre $n$ et annulé par $\Lambda_p(f(z))$. D'après le deuxième pas, la matrice compagnon de $\mathcal{L}_{1,\delta}$ est dans $\mathscr{M}_p$. Il découle du troisième pas que l'opérateur $\mathcal{L}_1$ est $p$-unipotent et finalement du cinquième pas que, tous les disques singuliers de $L_1(z)$ sont singuliers réguliers et $L_1(z)$ a exactement $r$ disques singuliers réguliers à distance finie. 
\end{proof}

\section{La série hypergéométrique $_2F_1(-1/2,1/2,1;16z)$}\label{exemple}
Nous allons montrer que la série \[\mathfrak{f}_2(z):={}_2F_1(-1/2,1/2,1;16z)=\sum_{n\geq0}\frac{-1}{2n-1}\binom{2n}{n}^2z^n\] appartient à $\mathcal{L}^2(\mathcal{P})\setminus\mathcal{L}(\mathcal{P})$, où $\mathcal{P}$ est l'ensemble des nombres premiers. Remarquons que la série $\mathfrak{f}_2(z)$ satisfait aux conditions du premier point du théorème~\ref{practico}. La série $\mathfrak{f}_2(z)$ annule l'opérateur différentiel $\mathcal{L}:=z(1-16z)\frac{d}{dz^2}+(1-16z)\frac{d}{dz}+4$. Cet opérateur est fuchsien, zéro est un point singulier régulier et les exposants de $\mathcal{L}$ en zéro sont égaux à zéro. Donc, l'opérateur $\mathcal{L}$ est MOM en zéro. De plus, d'après le théorème~6.2 de \cite{vmsff}, $\mathcal{L}$ a une structure de Frobenius forte pour tout nombre premier $p>2$ et finalement, il suit du lemme~\ref{P_0} que $\mathfrak{f}_2(z)\in 1+z\mathbb{Z}[[z]]$. Considérons la série hypergéométrique $f_1(z)={}_2F_1(1/2,1/2,1,16z)=\sum\limits_{n\geq0}\binom{2n}{n}^2z^n.$

\begin{lem}\label{P_0}
	La série $\mathfrak{f}_2(z)$ appartient à $1+z\mathbb{Z}[[z]]$ et, pour tout $p\in\mathcal{P}$, $$\mathfrak{f}_{2\mid p}(z)=P_{2,p}(z)(f_{1\mid p}(z))^p\quad\text{ et }\quad\mathfrak{f}_{2\mid p}(z)=\left(\frac{P_{1,p}(z)^p}{P_{2,p}(z)^{p-1}}\right)\mathfrak{f}_{2\mid p}(z)^p,$$ où $P_{1,p}(z)$ et $P_{2,p}(z)$ sont les $p$-troncatures de $f_1(z)$ et $\mathfrak{f}_2(z)$ respectivement.
\end{lem}

\begin{proof}
 Montrons que $(1-4z)^{-1/2}=\sum_{n\geq0}-\frac{1}{2n-1}\binom{2n}{n}z^n$ appartient à $1+z\mathbb{Z}[[z]]$.  Pour tout entier $n>0$, $\frac{1}{2n-1}\binom{2n}{n}=2C_{n-1}$, où $C_{n}:=\frac{1}{2n+1}\binom{2n+1}{n}$ est le nombre de Catalan de rang $n$. Il suit du corollaire~6.2.3 de \cite{ce} que pour tout entier $n\geq0$, $C_n\in\mathbb{N}$. Par conséquent, la série $(1-4z)^{-1/2}$ appartient à $1+z\mathbb{Z}[[z]]$. Mais il est clair que, pour tout $n\geq0$, $\binom{2n}{n}\in\mathbb{Z}$. Ainsi, pour tout $n\geq0$, $-\frac{1}{2n-1}\binom{2n}{n}^2\in\mathbb{Z}$. D'où,  $\mathfrak{f}_2(z)$ appartient à $1+z\mathbb{Z}[[z]]$. Il est bien connu que $f_{1\mid p}(z)=P_{1,p}(z)(f_{1\mid p}(z))^p$. Il est facile à vérifier que $f_{1}=\mathfrak{f}_2-2\delta\mathfrak{f}_2$. Alors, $f_1(z)+2\delta f_1(z)=f_2-4z(\mathfrak{f}'_2(z)+z\mathfrak{f}''_2(z))$. Mais, nous savons que $\frac{4}{1-16z}\mathfrak{f}_2+\mathfrak{f}'_2(z)+z\mathfrak{f}''_2(z)=0$. D'où, on obtient l'égalité suivante : 
 \begin{equation}\label{eq_p_1}
 \mathfrak{f}_{2}(z)=(1-16z)(f_1(z)+2\delta f_{1}(z))
 \end{equation}
 
 De manière similaire, on a 
 \begin{equation}\label{eq_p_2}
 P_{2,p}(z)=(1-16z)(P_{1,p}(z)+2\delta P_{1,p}(z))
 \end{equation}
 
 En appliquant l'opérateur $(1-16z)(1+2\delta)$ à $f_{1\mid p}(z)=P_{1,p}(z)(f_{1\mid p}(z))^p$, l'égalités~\eqref{eq_p_1} et \eqref{eq_p_2} impliquent que 
 \begin{equation}\label{eq_p_3}
 \mathfrak{f}_{2\mid p}(z)=P_{2,p}(z)(f_{1\mid p}(z))^p.
 \end{equation}
 Ainsi, on a 
 \begin{equation}\label{eq_p_4}
 \mathfrak{f}_{2\mid p}(z)=P_{2,p}(z)P_{1,p}(z)^pf_{1\mid p}(z)^{p^2}
 \end{equation}
 Mais, il suit de l'égalité~\eqref{eq_p_3} que $f_{1\mid p}^{p^2}=\left(\frac{ \mathfrak{f}_{2\mid p}(z)}{P_{2,p}(z)}\right)^p$. En remplaçant cette dernière égalité dans \eqref{eq_p_4}, on a $$\mathfrak{f}_{2\mid p}(z)=\left(\frac{P_{1,p}(z)^p}{P_{2,p}(z)^{p-1}}\right)\mathfrak{f}_{2\mid p}(z)^p.$$
\end{proof}

 Nous allons montrer que pour tout nombre premier $p\geq3$ les polynômes $P_{1,p}(z)$ et $P_{2,p}(z)$ sont premier entre eux.

\begin{lem}\label{2}
Pour tout nombre premier $p$, le polynôme $P_{2,p}(z)$ est séparable sur $\overline{\mathbb{F}_p}$. En particulier, $P_{2,p}$ et $P'_{2,p}$ n'ont pas de racine en commun.
\end{lem}

\begin{proof}
	Pour $p=2$, on a $P_{0,2}(z)=1$. Donc, le polynôme $P_{2,2}(z)$ est séparable sur $\overline{\mathbb{F}_2}$. Soit $p$ un nombre premier différent de 2 et soit $\mathcal{L}:=z(1-16z)\frac{d}{dz^2}+(1-16z)\frac{d}{dz}+4 $. La série $\mathfrak{f}_2(z)$ annule $\mathcal{L}$. D'après le lemme~\ref{P_0}, $\mathfrak{f}_{2\mid p}(z)=P_{2,p}(z)(f_{1\mid p}(z))^p$. Par conséquent, $\mathcal{L}_p(P_{2,p}(z))=0$. Soit $\alpha_i\in\overline{\mathbb{F}_p}$ une racine de $P_{2,p}(z)$. Écrivons $P_{2,p}(z)=b_0(z-\alpha_i)^{m_i}+\cdots+b_{r}(z-\alpha_i)^{\frac{p-1}{2}}$, où $b_0\neq0$. On veux montrer que $m_i=1$. Le développement limité de $z(1-16z)$ en $z-\alpha_i$ nous donne l'égalité suivante  \[z(1-16z)=(\alpha_i-16\alpha_i^2)+(1-32\alpha_i)(z-\alpha_i)-16(z-\alpha_i)^2\]
	et le développement limitée de $1-16z$ en $z-\alpha_i$ nous donne l'égalité suivante  \[1-16z=1-16\alpha_i-16(z-\alpha_i).\] Comme $P_{2,p}(z)$ annule l'opérateur différentiel $\mathcal{L}_p$ alors
	\begin{multline}\label{eq:exp}
	[(\alpha_i-16\alpha_i^2)+(1-32\alpha_i)(z-\alpha_i)-16(z-\alpha_i)^2]P''_{0,p}(z)\\+[1-16\alpha_i-16(z-\alpha_i)]P'_{0,p}(z)
	+4P_{0,p}(z)=0.
	\end{multline}
	 Mais, $$P'_{2,p}(z)=b_0m_1(z-\alpha_i)^{m_i-1}+\cdots+b_r\frac{p-1}{2}(z-\alpha_i)^{\frac{p-1}{2}-1}$$
	 et
	 \begin{multline*}
	 P''_{2,p}(z)=b_0m_i(m_i-1)(z-\alpha_i)^{m_i-2}+b_1m_i(m_i+1)(z-\alpha_i)^{m_i-1}\\+\cdots+b_r\frac{p-1}{2}(\frac{p-1}{2}-1)(z-\alpha_i)^{\frac{p-1}{2}-2}.
	 \end{multline*}
	  
	 Donc, il suit de \eqref{eq:exp} que,
	 
	 \begin{equation}\label{eqn:4}
	 a_{m_i-2}(z-\alpha_i)^{m_i-2}
	 +a_{m_i-1}(z-\alpha_i)^{m_i-1}
	 +\cdots+a_{\frac{p-1}{2}}(z-\alpha_i)^{\frac{p-1}{2}}=0,
	 \end{equation}
	  où $$a_{m_i-2}=(\alpha_i-16\alpha_i^2)b_0m_i(m_i-1)$$ et  $$a_{m_i-1}=(\alpha_i-16\alpha_i^2)b_1m_i(m_i+1)+(1-32\alpha_i)m_i(m_i-1)b_0+(1-16\alpha_i)m_ib_0.$$
	Il suit de \eqref{eqn:4} que $a_{m_i-2}=0$ et $a_{m_i-1}=0$. Par conséquent,
	\begin{equation}\label{eq1}
	(\alpha_i-16\alpha_i^2)m_i(m_i-1)b_0=0,
	\end{equation}
	\begin{equation}\label{eq2}
	(\alpha_i-16\alpha_i^2)b_1m_i(m_i+1)+(1-32\alpha_i)m_i(m_i-1)b_0+(1-16\alpha_i)m_ib_0=0.
	\end{equation}

	Comme $m_i\neq0$, $\alpha_i\neq0$ (zéro n'est pas une racine de $P_{2,p}(z)$) et $b_0\neq0$ alors il suit de \eqref{eq1} que, $(1-16\alpha_i)(m_i-1)=0$. Si $\alpha_i\neq1/16$ alors $m_i=1$. Maintenant si $\alpha_i=1/16$ alors l'égalité~\eqref{eq2} entraîne que $-m_i(m_i-1)=0$. Par conséquent, $m_i=1$.
\end{proof}

\begin{lem}
	Pour tout nombre premier $p\geq3$ les polynômes $P_{1,p}(z)$ et $P_{2,p}(z)$ sont premier entre eux.
\end{lem}
\begin{proof}
Soit $p$ un nombre premier différent de 2. Montrons que $P_{1,p}(z)$ et $P_{2,p}(z)$ n'ont pas de racine en commun. Comme $p$ est différent de 2, alors les polynômes $P_{1,p}(z)$ et $P_{2,p}(z)$ sont différents. L'égalité $P_{1,p}=P_{2,p}-2zP'_{2,p}$ entraîne tout de suite que $P_{1,p}(z)$ et $P_{2,p}(z)$ sont premier entre car, d'après le lemme~\ref{2}, $P_{2,p}$ et $P'_{2,p}$ n'ont pas de racine en commun.
\end{proof}

\begin{lem}
	Soit $\mathcal{S}$ un ensemble infini de nombres premiers. Alors la série $\mathfrak{f}_2(z)$ n'appartient pas à $\mathcal{L}(\mathcal{\mathcal{S}})$.
\end{lem}
\begin{proof}

Soit $p$ un nombre premier différent de 2. Nous posons $B_{0,p}=\frac{P_{1,p}^p}{P_{2,p}^{p-1}}$ et $B_{k,p}=B_0(z)B_0(z^p)\cdots B_0(z^{p^k})$. D'après le lemme~\ref{P_0}, il découle que $\mathfrak{f}_{2\mid p}(z)=B_k(z)\mathfrak{f}_{2\mid p}(z^{p^{k+1}}).$
On vérifie facilement que $B_k=\frac{P_{1,p}^{p+p^2+\cdots+p^{k+1}}}{P_{2,p}^{p^{k+1}-1}}$. Montrons que la hauteur de $B_k$ est $\frac{p}{2}(p^{k+1}-1)$. Comme $P_{1,p}(z)$ et $P_{2,p}(z)$ sont premiers entre eux, alors la fraction $B_k$ est réduite et sa hauteur est égale à $\frac{p-1}{2}(p+p^2+\cdots+p^{k+1})=\frac{p-1}{2}p(1+p+\cdots+p^{k})=\frac{p-1}{2}\frac{p}{p-1}(p^{k+1}-1)=\frac{p}{2}(p^{k+1}-1)$. Donc, pour tout nombre premier $p\geq3$, il existe une fraction rationnelle $A_k\in\mathbb{F}_p(z)$ réduite de hauteur $\frac{p}{2}(p^{k}-1)$ telle que $\mathfrak{f}_2(z)=A_k(z)\mathfrak{f}_2(z^{p^k})$. Par conséquent, $\mathfrak{f}_2(z)\notin\mathcal{L}(\mathcal{\mathcal{S}})$.
\end{proof}

\section{Équations de Calabi-Yau}\label{eqcalabi}
Dans \cite{calabi} les auteurs donnent une liste de plus de 400 opérateurs de type Calabi-Yau. Ces opérateurs vérifient certaines conditions algébriques dont: zéro est un point singulier régulier et les exposants en z\'ero sont tous égaux à zéro et chaque opérateur admet une solution dans $\mathbb{Z}[[z]]$ dont le terme constant est \'egal \`a 1. Notamment, tous ces opérateur sont MOM en zéro. Dans la plupart des cas, la série solution de ces op\'erateurs est donn\'ee dans \cite{calabi}. Grâce à la deuxième partie du théorème~\ref{practico}, nous avons montré dans la partie~\ref{210} que la série 210 de \cite{calabi} appartient à $\mathcal{L}(\mathcal{P}\setminus\mathcal{J})$, où $\mathcal{J}$ est un ensemble fini de nombres premiers. Nous soulignons que la stratégie utilisée pour montrer que la série 210 est dans $\mathcal{L}(\mathcal{P}\setminus\mathcal{J})$ peut être employée pour montrer que d'autres séries qui apparaissent dans \cite{calabi} appartiennent à $\mathcal{L}(\mathcal{S})$, où $\mathcal{S}$ est un ensemble de nombres premiers tel que $\mathcal{P}\setminus\mathcal{S}$ est fini. Illustrons encore une fois cette stratégie avec la série 26 de \cite{calabi}. Cette série est donnée par 
$$f(z)=\sum_{j\geq0}\left(\binom{2j}{j}\left(\sum_{k=0}^j\binom{j}{k}^2\binom{j+k}{k}\binom{2k}{j}\right)\right)z^j\in 1+z\mathbb{Z}[[z]].$$
D'après le théorème~3.5 de \cite{sb}, $f(z)$ est une diagonale d'une fraction rationnelle à coefficients dans $\mathbb{Q}$, alors de \cite{picardfuchs}, la série $f(z)$ annule un opérateur différentiel $\mathcal{H}\in\mathbb{Q}(z)[d/dz]$ muni d'une structure de Frobenius forte pour presque tout $p$. De plus, comme on l'a déjà mentionné, l'opérateur annulé par $f(z)$ décrit dans \cite{calabi} est MOM en zéro. Maintenant nous utilisons le th\'eor\`eme de Lucas pour montrer que pour tout nombre premier $p$, $\Lambda_p(f)_{\mid p}(z)=f_{\mid p}(z)$. En effet, d'après le théorème de Lucas $\binom{2jp}{jp}\equiv\binom{2j}{j}\mod p.$ Écrivons $k=s+lp$ avec $0\leq l\leq j$ et $0\leq s<p-1$. Supposons $s>0$, d'après le théorème de Lucas,
\begin{footnotesize}
$$\binom{jp}{s+lp}^2\binom{(j+l)p+s}{lp+s}\binom{2lp+2s}{pj}\equiv\binom{j}{l}^2\binom{0}{s}^2\binom{j+l}{l}\binom{s}{s}\binom{2lp+2s}{pj}\mod p=0$$
\end{footnotesize} car $s>0.$ Dans le cas $s=0$, d'apr\`es le the\'or\`eme de Lucas on obtient, $$\binom{jp}{lp}^2\binom{(j+l)p}{lp}\binom{2lp}{pj}\equiv\binom{j}{l}^2\binom{j+l}{l}\binom{2l}{j}\mod p.$$
Ainsi, 
\begin{small}
\begin{align*}
&\Lambda_p(f)_{\mid p}(z)=\sum_{j\geq0}\left(\binom{2jp}{jp}\left(\sum_{k=0}^{jp}\binom{jp}{k}^2\binom{jp+k}{k}\binom{2k}{jp}\right)\mod p\right)z^j\\
&=\sum_{j\geq0}\left(\binom{2jp}{jp}\left(\sum_{s=0}^{p-1}\sum_{l=0}^{j}\binom{jp}{s+lp}^2\binom{jp+s+lp}{s+lp}\binom{2lp+2s}{jp}\right)\mod p\right)z^j\\
&=\sum_{j\geq0}\left(\binom{2jp}{jp}\left(\sum_{l=0}^{j}\binom{jp}{lp}^2\binom{jp+lp}{lp}\binom{2lp}{jp}\right)\mod p\right)z^j\\
&=\sum_{j\geq0}\left(\binom{2j}{j}\left(\sum_{l=0}^{j}\binom{j}{l}^2\binom{j+l}{l}\binom{2l}{j}\right)\mod p\right)z^j\\
&=f_{\mid p}(z).
\end{align*}
\end{small}
Donc, la deuxième partie du théorème~\ref{practico} entraîne que la série $f(z)\in\mathcal{L}(\mathcal{P}\setminus\mathcal{J})$, où $\mathcal{J}$ est un ensemble fini de nombres premiers. En appliquant la même stratégie dont nous nous sommes servis pour montrer que les séries 26 et 210 de \cite{calabi} appartiennent à $\mathcal{L}(\mathcal{S})$, où $\mathcal{S}$ est un ensemble infini de nombres premiers tel que $\mathcal{P}\setminus\mathcal{S}$ est fini, nous obtenons qu'entre les séries qui apparaissent dans \cite{calabi}, il y en a 242 qui appartiennent \`a $\mathcal{L}(\mathcal{S})$, o\`u $\mathcal{P}\setminus\mathcal{S}$ est fini. \`A savoir: 

\textit{1-16}, \textit{18-25}, \textbf{26-28}, \textit{29}, \textit{\^1-$\widehat{\textit{14}}$}, \textit{30}, \textbf{32, 33}, \textit{35-41}, \textbf{42}, \textit{43-46}, \textit{48-55},  \textit{57-60}, \textit{62-70}, \textit{75-82}, \textit{85-92}, \textit{94, 95}, \textit{99-108}, \textbf{109}, \textit{110-115}, \textit{119-122}, \textbf{123}, \textit{124-129}, \textit{131,132}, \textit{146-153}, \textbf{154}, \textit{156-158}, \textit{160-165}, \textit{185}, \textbf{186}, \textbf{189}, \textit{190-192}, \textbf{193-196}, \textit{197}, \textbf{198}, \textbf{202}, \textit{208}, \textit{209}, \textbf{210}, \textit{212}, \textbf{213-219}, \textbf{222, 223}, \textbf{226}, \textit{232}, \textbf{234, 235}, \textit{238-240}, \textit{243}, \textbf{251}, \textit{284}, \textbf{286, 287}, \textit{292}, \textbf{293, 297, 298, 300, 301, 303, 304, 306}, \textit{307}, \textbf{308-316, 318, 320-322}, \textit{323}, \textbf{328, 329, 331-335, 336}, \textit{337, 338}, \textbf{339, 341-345, 348, 349, 359}, \textit{367}, \textit{369-371}, \textbf{373}, \textit{377}, \textbf{378, 379}, \textit{380}, \textbf{381, 385}, \textbf{394-396}, \textit{398}, \textbf{399-401}.  

Les cas en italique repr\'esentent les cas montr\'es dans \cite{Borisgfonct} et les autres sont les nouveaux cas.

\section{Indépendance algébrique}\label{independaalgebrique}
Rappelons que $f_1(z),\ldots, f_r(z)\in\mathbb{Q}[[z]]$ sont \emph{algébriquement dépendantes} sur $\mathbb{Q}(z)$ s'il existe un polynôme non nul $P(x_1,\ldots,x_r)$ à coefficients dans $\mathbb{Q}(z)$ tel que $P(f_1,\ldots, f_r)=0$. Et, $f_1(z),\ldots, f_r(z)$ sont \emph{algébriquement indépendantes} sur $\mathbb{Q}(z)$ si pour tout polynôme non nul $P(x_1,\ldots, x_r)$ à coefficients dans $\mathbb{Q}(z)$, $P(f_1,\ldots, f_r)\neq0$.
Notre but dans cette partie est de démontrer les deux théorèmes suivants. 
\begin{theoreme}\label{exempleind}
Soit la famille $\mathfrak{F}=\left\{\sum_{n\geq0}\frac{-1}{2n-1}\binom{2n}{n}^{r}z^n:r\geq2\right\}.$ 
	Alors tous les éléments de  $\mathfrak{F}$ sont algébriquement indépendants sur $\mathbb{Q}(z)$.
\end{theoreme}

\begin{theoreme}\label{melange}
	Les séries
	\begin{align*}
	\mathfrak{f}_2(z)=\sum_{n\geq0}\frac{-1}{2n-1}\binom{2n}{n}^{2}z^n\quad\textup{et}\quad \mathfrak{t}(z)=\sum_{n\geq0}\left(\sum_{k=0}^{n}\binom{n}{k}^2\binom{n+k}{k}^2\right)z^n
	\end{align*}
	sont algébriquement indépendantes sur $\mathbb{Q}(z)$.
\end{theoreme}
Nous avons montré dans la partie~\ref{exemple} que la série $\mathfrak{f}_2(z)$ n'appartient pas à $\mathcal{L}(\mathcal{P})$. Donc, le critère d'indépendance algébrique donné dans \cite{Borisgfonct} ne peut pas être appliqué aux séries données dans les théorèmes~\ref{exempleind} et \ref{melange}. La démonstration de ces deux théorèmes repose sur la proposition~\ref{indalg}. De plus, celle-ci nous fournit une stratégie pour montrer l'indépendance algébrique de certaines séries dans $\mathcal{L}^2(\mathcal{S})$.

 \begin{prop}\label{indalg}
	Soient $f_1(z),\ldots, f_r(z)\in\mathcal{F}(\mathcal{S})$, où $\mathcal{S}$ est un ensemble infini de nombres premiers. Soient $g_1(z),\ldots, g_r(z)\in 1+z\mathbb{Q}[[z]]$ telles que, pour tout $p\in\mathcal{S}$, $g_i(z)\in\mathbb{Z}_{(p)}[[z]]$. Supposons que $f_1(z),\ldots, f_r(z)$ annulent chacune un opérateur différentiel MOM en zéro à coefficients dans $\mathbb{Q}(z)$ et supposons que pour tout $p\in\mathcal{S}$ et tout $i\in\{1,\ldots,r\}$, il existe une entier strictement positif $l_{i,p}$ tel que $\Lambda^{l_{i,p}}_p(f_i(z))_{\mid p}=g_{i\mid p}(z)=\Lambda^{2l_{i,p}}_p(f_{i})_{\mid p}$. Si $g_1(z),\ldots, g_r(z)$ sont algébriquement indépendantes sur $\mathbb{Q}(z)$ alors $f_1(z),\ldots, f_r(z)$ sont algébriquement indépendantes sur $\mathbb{Q}(z)$.  
\end{prop}
 Sous les hypothèses de la proposition~\ref{indalg}, le théorème~\ref{practico} entraîne que pour chaque $i\in\{1,\ldots,r\}$, il existe un ensemble $\mathcal{S}_i\subset\mathcal{S}$ infini tel que $f_i(z)\in\mathcal{L}^2(\mathcal{S}_i)$ et $\mathcal{S}\setminus\mathcal{S}_i$ est fini. Donc, si $\mathcal{S}'=\mathcal{S}_1\cap\ldots\cap\mathcal{S}_r$ alors l'ensemble $\mathcal{S}'$  est infini, $\mathcal{S}\setminus\mathcal{S}'$ est fini et pour tout $i\in\{1,\ldots,r\}$, $f_i(z)\in\mathcal{L}^2(\mathcal{S}')$. Une observation importante dans la proposition~\ref{indalg} est que les séries $g_1(z),\ldots, g_r(z)$ sont dans $\mathcal{L}(\mathcal{S}_0)$, où $\mathcal{S}_0\subset\mathcal{S}'$ est infini. Cette observation est montré dans le lemme~\ref{periode} ci-dessous. Donc, la proposition~\ref{indalg} nous fournit un critère de transfert d'indépendance algébrique des séries qui appartiennent à $\mathcal{L}(\mathcal{S}_0)$ aux séries qui sont dans $\mathcal{L}^2(\mathcal{S}_0)$.

\begin{lem}\label{periode}
	Soit $\mathcal{S}$ un ensemble infini de nombres premiers et soient $f(z)\in\mathcal{F}(\mathcal{S})$ et $g(z)\in 1+z\mathbb{Q}[[z]]$ tels que, pour tout $p\in\mathcal{S}$, $g(z)\in\mathbb{Z}_{(p)}[[z]]$. Supposons que $f(z)$ annule un opérateur différentiel MOM en zéro à coefficients dans $\mathbb{Q}(z)$. Si, pour tout $p\in\mathcal{S}$, il existe un entier strictement positif $l_p$ tel que $\Lambda^{2l_p}_p(f_{\mid p}(z))=g_{\mid p}(z)=\Lambda^{l_p}_p(f_{\mid p}(z))$ alors, il existe un ensemble infini $\mathcal{S}_0\subset\mathcal{S}$ tel que l'ensemble $\mathcal{S}\setminus\mathcal{S}_0$ est fini et la série $g(z)$ appartient à $\mathcal{L}(\mathcal{S}_0)$. De plus, pour tout $p\in\mathcal{S}_0$, il existe une fraction rationnelle $B_p(z)\in\mathbb{F}_p(z)\cap\mathbb{F}_p[[z]]$ telle que $f_p(z)=B_p(z)g_{\mid p}(z)$.
\end{lem}

\begin{proof}
	Comme $f(z)$ est dans $\mathcal{F}(\mathcal{S})$ et annule un opérateur différentiel MOM en zéro à coefficients dans $\mathbb{Q}(z)$, d'après le lemme~\ref{ordre1}, il existe un ensemble infini $\mathcal{S}_0\subset\mathcal{S}$ tel que: l'ensemble $\mathcal{S}\setminus\mathcal{S}_0$ est fini et pour tout $p\in\mathcal{S}_0$ et tout couple d'entiers positifs $(i,m)$, il existe une fraction rationnelle $A_{p,i,m}(z)\in\mathbb{F}_p(z)\cap\mathbb{F}_p[[z]]$ de hauteur inférieure ou égale à $Cp^m$ telle que $\Lambda^i_p(f_{\mid p}(z))=A_{p,i,m}(z)(\Lambda^{i+m}_p(f_{\mid p}(z)))^{p^m}$, où $C$ ne dépend pas de $p$. En particulier, pour $p\in\mathcal{S}_0$ et les entiers strictement positifs $i=l_p=m$ il existe une fraction rationnelle $A_{p,l_p,l_p}(z)\in\mathbb{F}_p(z)\cap\mathbb{F}_p[[z]]$ de hauteur inférieure ou égale à $Cp^{l_p}$ telle que $\Lambda^{l_p}_p(f_{\mid p}(z))=A_{p,l_p,l_p}(z)(\Lambda^{2l_p}_p(f_{\mid p}(z))^{p^{l_p}}$. Mais par hypothèse, $\Lambda^{2l_p}_p(f_{\mid p}(z))=g_{\mid p}(z)=\Lambda^{l_p}_p(f_{\mid p}(z))$ donc, $g_{\mid p}(z)=A_{p,l_p,l_p}(z)g_{\mid p}(z^{p^{l_p}})$. Par conséquent, la série $g(z)$ appartient à $\mathcal{L}(\mathcal{S}_0)$. Maintenant, d'après le lemme~\ref{ordre1}, pour $p\in\mathcal{S}_0$ il existe une fraction rationnelle $A_{p,0,l_p}(z)\in\mathbb{F}_p(z)\cap\mathbb{F}_p[[z]]$ de hauteur inférieure ou égale à $Cp^{l_p}$ telle que $f_{\mid p}(z)=A_{p,0,l_p}(z)(\Lambda^{l_p}_p(f_{\mid p}(z)))^{p^{l_p}}$. Mais par hypothèse, $\Lambda^{l_p}_p(f_{\mid p}(z))=g_{\mid p}(z)$ alors $f_{\mid p}(z)=A_{p,0,l_p}(z)g_{\mid p}(z^{p^{l_p}})$. Donc, $f_{\mid p}(z)=\frac{A_{p,0,l_p}(z)}{A_{p,l_p,l_p}(z)}g_{\mid p}(z)$. Le terme constant de $A_{p,l_p,l_p}(z)$ est égal à 1 car le terme constant des séries $\Lambda^{l}_p(f_{\mid p}(z))$, $\Lambda^{2l_p}_p(f_{\mid p}(z))$ est 1. Donc, $\frac{A_{p,0,l_p}(z)}{A_{p,l_p,l_p}(z)}\in\mathbb{F}_p(z)\cap\mathbb{F}_p[[z]]$.
\end{proof}

Maintenant nous appliquerons la proposition~\ref{indalg} pour montrer les théorèmes~\ref{exempleind} et \ref{melange}.
\begin{proof}[Démonstration de \ref{exempleind}]
	 Pour chaque $r\geq2$ on pose
	
	\[\mathfrak{f}_{r}=\sum_{n\geq0}\frac{-1}{2n-1}\binom{2n}{n}^{r}z^n \quad\textup{et}\quad\mathfrak{g}_{r}=\sum_{n\geq0}\binom{2n}{n}^{r}z^n.\]
	Soit $\mathcal{P}$ l'ensemble des nombres premiers. Remarquons d'abord que $\mathfrak{f}_{r}(z)$ et $\mathfrak{g}_r(z)$ appartiennent à $1+z\mathbb{Z}[[z]]$. D'après le théorème de Lucas, pour tout nombre premier $p$ et tout $r\geq2$, $\Lambda_p(\mathfrak{f}_{r}(z))_{\mid p}=\mathfrak{g}_{r\mid p}(z)=\Lambda^2_p(\mathfrak{f}_{r}(z))_{\mid p}$. Nous allons d'abord montrer que pour $r\geq2$ la série $\mathfrak{f}_{r}(z)$ est dans $\mathcal{F}(\mathcal{P}\setminus\{2\})$. La série $\mathfrak{f}_{r}(z)$ annule l'opérateur différentiel $$\mathcal{L}_{r}=\delta^{r}-4^{r}z(\delta-1/2)(\delta+1/2)^{r-1},$$ où $\delta=z\frac{d}{dz}$. Grâce au théorème~6.2 de \cite{vmsff}, l'opérateur $\mathcal{L}_{r}$ est muni d'une structure de Frobenius forte pour tout nombre premier $p$ différent de 2. Ainsi, $\mathfrak{f}_{r}(z)\in\mathcal{F}(\mathcal{P}\setminus\{2\})$. De plus, l'opérateur $\mathcal{L}_r$ est MOM en zéro. Soit $l\geq2$ et montrons que $\mathfrak{f}_{2}(z),\ldots,\mathfrak{f}_{l}(z)$ sont algébriquement indépendantes sur $\mathbb{Q}(z)$. Notons que $\mathfrak{f}_{2}(z),\ldots,\mathfrak{f}_{l}(z)\in\mathcal{F}(\mathcal{P}\setminus\{2\})$ et, pour tout $p\in\mathcal{P}\setminus\{2\}$ et tout $r\in\{2,\ldots, l\}$, on a que $\Lambda_p(\mathfrak{f}_{r}(z))_{\mid p}=\mathfrak{g}_{r\mid p}(z)=\Lambda^2_p(\mathfrak{f}_{r}(z))_{\mid p}$. Donc, d'après la proposition~\ref{indalg}, pour montrer que $\mathfrak{f}_{2}(z),\ldots,\mathfrak{f}_{l}(z)$ sont algébriquement indépendantes sur $\mathbb{Q}(z)$, il suffit de montrer que $\mathfrak{g}_{2}(z),\ldots,\mathfrak{g}_{l}(z)$ sont algébriquement indépendantes sur $\mathbb{Q}(z)$. En effet, à la suite du théorème~2.1 de \cite{Borisgfonct}, les séries $\mathfrak{g}_{1}(z),\ldots,\mathfrak{g}_{l}(z)$ sont algébriquement indépendantes sur $\mathbb{Q}(z)$. Par conséquent, pour tout $l\geq2$ les séries $\mathfrak{f}_{2}(z),\ldots,\mathfrak{f}_{l}(z)$ sont algébriquement indépendantes sur $\mathbb{Q}(z)$, d'où tous les éléments de $\mathfrak{F}$ sont algébriquement indépendants sur $\mathbb{Q}(z)$. 
\end{proof}

Avant de faire la démonstration du théorème~\ref{melange} nous montrons que $\mathfrak{t}(z)$ est dans $\mathcal{F}(\mathcal{S})$, où $\mathcal{S}$ est un ensemble infini de nombres premiers et, annule un opérateur différentiel MOM en zéro. D'après le théorème~3.5 de \cite{sb}, la série $\mathfrak{t}(z)$ est la diagonale d'une fraction rationnelle, alors de \cite{picardfuchs}, la série $\mathfrak{t}(z)$ annule un opérateur $\mathcal{H}\in\mathbb{Q}(z)[d/dz]$ muni d'une structure de Frobenius forte pour presque tout $p$. D'autre part, l'opérateur différentiel $$\mathcal{D}:(1-34z+z^2)z^2\frac{dz}{dz^3}+(3-153z+6z^2)z\frac{d}{dz^2}+(1-112z+7z^2)\frac{d}{dz}-5+z$$
est annulé par la série $\mathfrak{t}(z)$. Cet opérateur est MOM en zéro.

\begin{proof}[Démonstration du théorème \ref{melange}]
	 Considérons $\mathfrak{g}_2(z)=\sum_{n\geq0}\binom{2n}{n}^2z^n$. Donc pour tout $p\in\mathcal{P}$, $\Lambda_p(\mathfrak{f}_2(z))_{\mid p}=\mathfrak{g}_{2\mid p}(z)=\Lambda^2_p(\mathfrak{f}_2(z))_{\mid p}$ et, $\Lambda_p(\mathfrak{t}(z))_{\mid p}=\mathfrak{t}(z)=\Lambda^2_p(\mathfrak{t}(z))_{\mid p}$. Alors, d'après la proposition~\ref{indalg}, il suffit de montrer que $\mathfrak{g}_2(z)$ et $\mathfrak{t}(z)$ sont algébriquement indépendantes sur $\mathbb{Q}(z)$. Il découle du lemme~\ref{periode} qu'il existe un ensemble infini $\mathcal{S}''$ de nombres premiers tel que $\mathfrak{g}_2(z),\mathfrak{t}(z)\in\mathcal{L}(\mathcal{S}'')$. À la suite du théorème~5.1 de \cite{Borisgfonct}, si  $\mathfrak{g}_2(z)$ et $\mathfrak{t}(z)$ sont algébriquement dépendantes sur $\mathbb{Q}(z)$ il existe des entiers $a,b$ non tous nuls et une fraction rationnelle $r(z)\in\mathbb{Q}(z)$ tels que $\mathfrak{g}_2(z)^a\mathfrak{t}(z)^b=r(z).$
	Alors, 
	\begin{equation}\label{t,g}
	\mathfrak{t}(z)^b=r(z)\mathfrak{g}_2(z)^{-a}.
	\end{equation}
	Comme $\mathfrak{t}(z)$ annule l'opérateur différentiel $\mathcal{D}$ alors le rayon de convergence de $\mathfrak{t}(z)$ est égal à $\rho_{\mathfrak{t}}=17-12\sqrt{2}$ et le rayon de convergence de $\mathfrak{g}_2(z)$ est égal à $\rho_{\mathfrak{g}_2}=1/16$ car elle annule l'opérateur $z(1-16z)\frac{d}{dz^2}+(1-32z)\frac{d}{dz}-4$. Comme $\rho_{\mathfrak{t}}<\rho_{\mathfrak{g}_2}$ donc, d'après \eqref{t,g}, la série $\mathfrak{t}(z)^b$ est méromorphe au voisinage de $z_0=17-12\sqrt{2}$. Mais, il est montré dans \cite[p.~555]{Borisgfonct} que pour tout entier $c$ différent de zéro la série $\mathfrak{t}(z)^c$ n'admet pas une continuation méromorphe au voisinage de $z_0$. Par conséquent, $b=0$ et ainsi $\mathfrak{g}_2(z)^a=r(z)$. Notamment, comme $a\neq0$ alors $\mathfrak{g}_2(z)$ est algébrique sur $\mathbb{Q}(z)$. Mais, d'après l'article \cite{transcedencia} de Sharif et Woodcock, on sait que la série $\mathfrak{g}_2(z)$ est transcendante sur $\mathbb{Q}(z)$. Ce qui nous amène à une contradiction. Alors $\mathfrak{g}_2(z)$ et $\mathfrak{t}(z)$ sont algébriquement indépendantes sur $\mathbb{Q}(z)$. Ainsi, la proposition~\ref{indalg} entraîne que $\mathfrak{f}_2(z)$ et $\mathfrak{t}(z)$ sont algébriquement indépendantes sur $\mathbb{Q}(z)$.
\end{proof}
Nous finissons cette partie en démontrant la proposition~\ref{indalg}.
Cette démonstration s'appuie fortement sur les techniques développées par Adamczewski, Bell et Delaygue dans \cite{Borisgfonct}.

\begin{proof}[Démonstration de la proposition~\ref{indalg}]
	D'après le théorème~\ref{practico}, pour chaque $i\in\{1,\ldots,r\}$, la série $f_i\in\mathcal{L}^2(\mathcal{S}_i)$, où  $\mathcal{S}\setminus\mathcal{S}_i$ est fini. Soit $\mathcal{S}'=\mathcal{S}_1\cap\cdots\cap\mathcal{S}_r$, alors $\mathcal{S}'$ est infini et $f_1(z),\ldots, f_r(z)\in\mathcal{L}^2(\mathcal{S}')$. Supposons que $f_1(z),\ldots, f_r(z)$ sont algébriquement dépendantes sur $\mathbb{Q}(z)$. Donc, il existe un polynôme $P(x_1,\ldots,x_r)\in\mathbb{Q}(z)[x_1,\ldots, x_r]$ non nul  de degré total $d$ tel que $P(f_1,\ldots, f_r)=0$. 
D'après le lemme~\ref{periode}, pour chaque $i\in\{1,\ldots,r\}$, la série $g_i(z)\in\mathcal{L}(\mathcal{S}'_i)$, où $\mathcal{S}'\setminus\mathcal{S}'_i$ est fini. Soit $\mathcal{S}''=\mathcal{S}'_1\cap\cdots\cap\mathcal{S}'_r$, alors $\mathcal{S}''$ est infini et $g_1(z),\ldots, g_r(z)\in\mathcal{L}(\mathcal{S}'')$. De plus, encore par le lemme~\ref{periode}, si $p\in\mathcal{S}''$ alors, il existe une fraction rationnelle $B_i(z)\in\mathbb{F}_p(z)\cap\mathbb{F}_p[[z]]$ telle que $f_{i\mid p}(z)=B_i(z)g_{i\mid p}(z)$. 
	
Soit $\mathcal{S}^{(3)}\subset\mathcal{S}''$ tel que pour tout $p\in\mathcal{S}^{(3)}$, le polynôme $P_{\mid p}$ est non nul dans $\mathbb{F}_p[z][x_1,\ldots, x_r]$. L'ensemble $\mathcal{S}^{(3)}$ est infini. Écrivons le polynôme $P_{\mid p}=\sum_{(i_1,\ldots,i_r)\in\mathbb{N}^r}a_{(i_1,\ldots,i_r)}(z)x^{i_1}_1\cdots x^{i_r}_r$, où $a_{(i_1,\ldots,i_r)}(z)\in\mathbb{F}_p[z]$. Mais, $f_{i\mid p}(z)=B_i(z)g_{i\mid p}(z)$, où $B_i(z)\in\mathbb{F}_p(z)\cap\mathbb{F}_p[[z]]$, alors $g_{1\mid p}(z),\ldots, g_{r\mid p}(z)$ annulent le polynôme $$\sum_{(i_1,\ldots,i_r)\in\mathbb{N}^r}a_{(i_1,\ldots,i_r)}(z)B_1(z)^{i_1}\cdots B_r(z)^{i_r}x^{i_1}_1\cdots x^{i_r}_r$$
dont le degré total est inférieur ou égal à $d$ et $a_{(i_1,\ldots,i_r)}(z)B_1(z)^{i_1}\cdots B_r(z)^{i_r}\in\mathbb{F}_p(z)\cap\mathbb{F}_p[[z]]$.  Alors, pour tout $p\in\mathcal{S}^{(3)}$, les séries $g_{1\mid p}(z),\ldots, g_{r\mid p}(z)$ annulent un polynôme non nul à coefficients dans $\mathbb{F}_p(z)$ de degré total inférieur ou égal à $d$. Puisque $\mathcal{S}^{(3)}\subset\mathcal{S}''$ et $g_i(z)\in\mathcal{L}(\mathcal{S}'')$ pour tout $i\in\{1,\ldots,r\}$, alors pour tout $p\in\mathcal{S}^{(3)}$ et tout $i\in\{1,\ldots,r\}$, il existe  $A_i(z)\in\mathbb{F}_p(z)\cap\mathbb{F}_p[[z]]$ et un entier $k_i$ strictement positif tel que $g_{i\mid p}(z)=A_i(z)g_{i\mid p}(z^{p^{k_i}})$, où la hauteur de $A_i(z)$ est inférieure ou égale à $C_ip^{k_i}$. Si $k=k_1\cdots k_r$ et $C=2\max\{C_1,\ldots, C_r\}$ alors, d'après la remarque~4.2 de \cite{Borisgfonct}, pour tout $i\in\{1,\ldots,r\}$ il existe $Q_i(z)\in\mathbb{F}_p(z)\cap\mathbb{F}_p[[z]]$ tel que $g_{i\mid p}(z)=Q_i(z)g_{i\mid p}(z^{p^{k}})$, où la hauteur de $Q_i(z)$ est inférieure ou égale à $Cp^k$. Maintenant, nous allons appliquer la proposition~5.3 de \cite{Borisgfonct} à $L$ le corps de fractions de $\mathbb{F}_p[[z]]$, $M=\mathbb{F}_p(z)$ et l'endomorphisme injectif de $L$ défini par $\sigma(g(z))=g(z^{p^k})=g(z)^{p^k}$. Alors la proposition~5.3 de \cite{Borisgfonct} implique qu'il existe $m_1,\ldots, m_r\in\mathbb{Z}$ non tous nuls et $r(z)\in\mathbb{F}_p(z)$ tels que $$Q_1(z)^{m_1}\cdots Q_r(z)^{m_r}=\frac{r(z)^{p^k}}{r(z)}=r(z)^{p^k-1},$$
avec $|m_1+\cdots+m_r|<d$ et $|m_i|<d$. Comme le terme constant des $Q_i(z)$ est 1 alors $Q_i(z)^{m_i}\in\mathbb{F}_p[[z]]$ et ainsi, $r(z)\in\mathbb{F}_p[[z]]$ et $r(0)=1$. Notons que $Q_1(z),\ldots, Q_r(z)$, $r(z)$ et $m_1,\ldots, m_n$ dépendent de $p$. Comme les $m_i$ appartiennent à un ensemble fini, le principe des tiroirs implique qu'il existe $\mathcal{S}^{(4)}\subset\mathcal{S}^{(3)}$ infini et des entiers $t_1,\ldots, t_r$ tels que, pour tout $p\in\mathcal{S}^{(4)}$, $m_i=t_i$ pour $i\in\{1,\ldots,r\}$. Supposons que $p\in\mathcal{S}^{(4)}$ et écrivons $r(z)=\frac{a(z)}{b(z)}$, où $a(z),b(z)\in\mathbb{F}_p[z]$ sont premiers entre eux. Comme la hauteur de $Q_i$ est inférieure ou égale à $Cp^k$ alors la hauteur de $Q^{t_1}_1\cdots Q^{t_r}_r$ est inférieure ou égale à $Cp^k(|t_1|+\cdots+|t_r|)$. Et comme $Q^{t_1}_1\cdots Q^{t_r}_r=r(z)^{p^k-1}$
	alors la hauteur de $r(z)$ est inférieure ou égale à $$C\frac{p^k}{p^k-1}(|t_1|+\cdots+|t_r|)<2Crd.$$
	D'où les degrés de $a(z)$ et $b(z)$ sont inférieurs ou égaux à $2Crd$. On pose $$h(z)=g_1(z)^{-t_1}\cdots g_{r}(z)^{-t_r}\in\mathbb{Q}[[z]].$$ 
	Donc, pour tout $p\in\mathcal{S}^{(4)}$ on a que \begin{align*}
	h_{\mid p}(z^{p^k})=g_{1\mid p}(z)^{-t_1}\cdots g_{r\mid p}(z)^{-t_r}Q_1^{t_1}\cdots Q_r^{t_r}=h_{\mid p}(z)r(z)^{p^{k}-1}.
	\end{align*}
	Comme $h(0)=1$ alors $h_{\mid p}(z)$ est non nul et ainsi $h_{\mid p}(z)^{p^k-1}=r(z)^{p^{k}-1}$. Par conséquent, $\frac{r(z)}{h_{\mid p}(z)}$ est solution du polynôme $x^{p^k-1}-1$. Comme $r(0)=1$ et $h_{\mid p}(0)=1$ alors $\frac{r(z)}{h_{\mid p}(z)}=1$. Maintenant, on sait que pour tout $p\in\mathcal{S}^{(4)}$, $a(z)=a_0+a_1z+\cdots+a_iz^i$ et $b(z)=b_0+b_1z\cdots+b_jz^j$, où $0\leq i,j\leq 2Crd$. Alors, $$\sum_{j=0}^{2Crd}b_jz^jh_{\mid p}(z)-\sum_{i=0}^{2Crd}a_iz^i=0.$$ Par conséquent, $h_{\mid p}, zh_{\mid p}(z),\ldots, z^{2Crd}h_{\mid p}, z,\ldots, z^{2Crd}$ sont linéairement dépendantes sur $\mathbb{F}_p$. Comme $\mathcal{S}^{(4)}$ est infini, d'après le lemme~5.4 de \cite{Borisgfonct}, les séries $zh(z),\ldots, z^{2Crd}h(z), z,\ldots, z^{2Crd}$ sont linéairement dépendantes sur $\mathbb{Q}$, d'où $h(z)\in\mathbb{Q}(z)$. Alors, les séries $g_1(z),\ldots, g_r(z)$ sont algébriquement dépendantes sur $\mathbb{Q}(z)$.
\end{proof}



\begin{thebibliography}{99}
	
	\bibitem{AB13} 
	{\sc B. Adamczewski and J. Bell,} {\it Diagonalization and rationalization of algebraic Laurent series,} 
	Ann. Sci. \'Ec. Norm. Sup\'er \textbf{46} (2013), 963--1004.
	
	\bibitem{Borisgfonct}
	{\sc B. Adamczewski, J. Bell, and E. Delaygue,}
	{\it Algebraic independence of G-functions and congruences "à la Lucas'',}
	Ann. Sci. \'Ec. Norm. Sup\'er \textbf{52} (2019), 515--559.
	
	\bibitem{AY}
	{\sc B. Adamczewski and R. Yassawi,}
	{\it A note on Christol 's theorem,}
	preprint  2019, arXiv:1906.08703, 14 pp.

	
\bibitem {allouche}
{\sc J--P. Allouche, D. Gouyou-Beauchamps, and G. Skordev,}
{\it Transcendence of binomial and Lucas' formal power series,} J. Algebra \textbf{210} (1998), 577--592.


\bibitem {calabi}
{\sc G. Almkvist, C. Van Eckenvort, D. Van Straten, and W. Zudilin,} 
{\it Tables of Calbi-Yau equations,}
prepint 2010, arXiv:0507430, 130 pp.


\bibitem {sb}
{\sc A. Bostan, P. Lairez, and B. Salvy,}
{\it Multiple binomial sums,}
J. Symbolic Comput \textbf{80} (2017), 351--386.
 

\bibitem {pautomata}
{\sc G. Christol,}
{\it Ensembles presque périodiques $k$-reconnaissables,}
Theoret. Comput. Sci \textbf{9} (1979), 141--145. 

\bibitem {gillesff}
{\sc G. Christol,}
{\it Systèmes différentiels linéaires $p$-adiques, structures de Frobenius faible,}
Bull. Soc. Math. France \textbf{109} (1981), 83--122.

\bibitem {Gillesmoduldiff}
{\sc G. Christol,}
{ Modules différentiels et équations différentielles $p$-adiques,}
Queen's papers in pure and applied mathematics \textbf{66}, Queen's University, Kingston, 1983.


\bibitem {christolpadique}
{\sc G. Christol,}
{\it Un théorème de transfert pour les disques singuliers réguliers,}
Astérisque \textbf{119-120} (1984), 151--168.

\bibitem {picardfuchs}
{\sc G. Christol,} 
{\it Diagonales de fractions rationnelles et équations de Picard--Fuchs,}
Study group on ultrametric analysis \textbf{12}  (1984/85), Exp. No 13, 12 pp.


\bibitem {Gillesalgebriques}
{\sc G. Christol,}
{\it Fonctions et \'el\'ements alg\'ebriques,}
Pacific J. Math \textbf{125} (1986), 1--37. 



\bibitem {gilleglobalborne}
{\sc G. Christol,}
Globally bounded solutions of differential equations,
Analytic number theory (Tokyo, 1988), pp 45–64,
Lecture Notes in Math \textbf{1434}, Springer, Berlin, 1990.




\bibitem {dworksFf}
{\sc B.M Dwork,}
{\it On $p$-adic differential equations I. The Frobenius structure of differential equations,}
Bull. Soc. Math. France \textbf{39-40} (1974), 27--37.

\bibitem{Dworkgfunciones}
{\sc B. Dwork, G. Gerotto, And F. Sullivan,}
{An introduction to G-functions,}
Annals of Mathematics Studies \textbf{133}, Princeton University Press, 1994.







\bibitem{eilenberg}
{\sc S. Eilenberg,}
{Automata, languages, and machines Vol. A,}
Pure and Applied Mathematics \textbf{58}, Academic Press, New York, 1974.

\bibitem{gessel}
{\sc I. Gessel,}
{\it Some congruences for Apéry numbers,}
J. Number Theory \textbf{14} (1982), 362--368.

\bibitem {honda}
{\sc T. Honda,}
{\it Algebraic differential equations,}
Symposia Mathematica, Vol. XXIV (Rome, 1979), pp. 169–204, Academic Press, London-New York, 1981. 


\bibitem {apery}
{\sc A. Malik And A. Straub,}
{\it Divisibility properties of sporadic Apéry-like numbers,}
Res. Number Theory \textbf{2} (2016), 26 pp. 

\bibitem {lucas}
{\sc R. Mestrovic,}
{\it Lucas’ theorem: its generalizations, extensions and applications (1878–2014),}
preprint 2014, arXiv:1409.3820, 51 pp.


\bibitem{Singer}
{\sc M. van der Put and M. F. Singer,}
{Galois theory of linear differential equations,}
Grundlehren der Mathematischen Wissenschaften \textbf{328}. Springer-Verlag, Berlin, 2003.


\bibitem {ce}
{\sc R. P. Stanley,}
{Enumerative combinatorics Vol. 2,}
Cambridge Studies in Advanced Mathematics \textbf{62}, Cambridge University Press, Cambridge, 1999.


\bibitem{vmsff}
{\sc D. Vargas-Montoya,}
{\it Algébricité modulo $p$, séries hypergéométriques et structure de Frobenius forte,}
 Bull. Soc. Math. France \textbf{149 }(2021), 439--477. 


\bibitem {transcedencia}
{\sc F. Woodcock, H. Sharif,}
{\it On the transcendence of certain series,}
J. Algebra \textbf{121} (1989), 364--369.



	
	
\end{thebibliography}
\end{document}